\documentclass[]{article}

\usepackage{amsmath,amssymb,amsfonts,amsthm,mathrsfs,paralist,amsrefs}
\usepackage{mathtools,enumitem,color}
\usepackage{setspace}
\usepackage{authblk}
\usepackage[dvipsnames]{xcolor}
\theoremstyle{plain}
\newtheorem{thm}{Theorem}[section]
\newtheorem{prop}[thm]{Proposition}
\newtheorem{lemma}[thm]{Lemma}
\newtheorem{cor}[thm]{Corollary}

\theoremstyle{definition}
\newtheorem{defi}[thm]{Definition}

\theoremstyle{remark}
\newtheorem{remark}[thm]{Remark}

\newcommand{\C}{\mathbb {C}}
\newcommand{\R}{\mathbb {R}}

\newcommand{\N}{\mathbb {N}}

\def\BFa {\mathbf{a}}
\def\BFe {\mathbf{e}}
\def\BFb {\mathbf{b}}

\def\BFu {\mathbf{u}}
\def\BFU {\mathbf{U}}
\def\BFV {\mathbf{V}}
\def\BFv {\mathbf{v}}
\def\BFx {\mathbf{x}}
\def\BFX {\mathbf{X}}
\def\BFy {\mathbf{y}}
\def\BFY {\mathbf{Y}}
\def\BFi {\textit{\bfseries i}}
\def\BFk {\textit{\bfseries k}}

\def\BFz {\mathbf{z}}

\DeclareMathOperator{\vol}{Vol}

\def\hjaspar#1{}
\def\hchristiane#1 {}

\begin{document}

  \title{On the dependence structure and quality of scrambled $(t,m,s)$-nets}
\author[1]{Jaspar Wiart}
\author[2]{Christiane Lemieux\thanks{Corresponding author}}
\author[2]{Gracia Y.\ Dong}

 \affil[1]{Johannes Kepler University \authorcr Altenbergerstr.\ 69 \authorcr
4040 Linz, Austria \authorcr  jaspar.wiart@jku.at}

\affil[2]{Department of Statistics and Actuarial Science\authorcr
University of Waterloo, \authorcr 200 University Avenue West \authorcr Ontario, Canada, N2L 3G1 \authorcr email: clemieux@uwaterloo.ca, gracia.dong@uwaterloo.ca }

 \date{}
\maketitle

\begin{abstract}
 In this paper we develop a framework 
to study the dependence structure of scrambled $(t,m,s)$-nets. It relies on values denoted by $C_b(\BFk;P_n)$, which are related to how many distinct pairs of points from  $P_n$ lie in the same elementary $\BFk-$interval in base $b$. These values  quantify the equidistribution properties of $P_n$ in a more informative way than the parameter $t$. They also play a key role in determining if a scrambled set $\tilde{P}_n$ is    negative lower orthant dependent (NLOD). Indeed  this property holds if and only if $C_b(\BFk;P_n) \le 1$ for all $\BFk \in \N^s$, which in turn implies that a scrambled digital $(t,m,s)-$net in base $b$ is NLOD 
 if and only if $t=0$. Through numerical examples we demonstrate that these  $C_b(\BFk;P_n)$ values are a powerful tool to compare the quality of different $(t,m,s)$-nets, and to enhance our understanding of how scrambling can improve the quality of deterministic point sets. 
 \end{abstract}
 
   
\section{Introduction}

Quasi-Monte Carlo methods rely on low-discrepancy point sets and sequences to construct estimates for multidimensional integrals over the unit hypercube. In this context, the notion of discrepancy refers to the distance between the uniform distribution and the empirical distribution induced by a point set. This measure of non-uniformity is particularly suitable for deterministic point sets, for which a number of results exist that provide asymptotic results on the discrepancy of various constructions, including digital $(t,m,s)$-nets \cite{DiPi10,rNIE92b}.

In recent years, the use of randomized quasi-Monte Carlo methods has gained in popularity. By introducing randomness in a low-discrepancy point set, one gains not only access to probabilistic error estimates, but also in some cases to an improvement in the uniformity of the point set. In particular, the scrambled digital nets introduced by Owen in 1995 \cite{vOWE95a} have been used in different applications in practice. A number of results studying the variance of the corresponding estimators have been proved: see, for example, \cite{vOWE03a,Ger15a}.  For smooth enough functions, results in  \cite{vOWE97b} show a much better convergence rate for the variance of these scrambled net estimators than the Monte Carlo equivalent. Other results give  bounds holding for all square-integrable functions, where the scrambled net variance is shown to be no larger than a constant (larger than one, and possibly quite large depending on the net) times the Monte Carlo variance \cite{vOWE97a,vOWE97c}.

In \cite{Lem17}, a new approach to study scrambled $(0,m,s)$-nets was  introduced. It is based on the concept of negative lower/upper orthant dependence, and how it can be used to study the covariance term that differentiates the variance of Monte Carlo sampling-based estimators from that of scrambled $(0,m,s)$-nets. To study this covariance term, a new representation result was used. It is based on multivariate integration by parts, which allows to decompose the covariance term in a part that assesses the underlying point set---via its dependence structure---and a part that depends on the function. It is worth noting that a potentially larger class of functions than those of bounded variation in the sense of Hardy and Krause \cite{rNIE92b} can be studied via this decomposition. In the same paper, it was proved that two-dimensional scrambled $(0,m,2)$-nets have a variance no larger than a Monte Carlo estimator for functions that are monotone in each variable. This result was obtained by first establishing that scrambled $(0,m,2)$-nets are  negatively lower orthant dependent.

Following \cite{Lem17}, a number of other authors have pursued the idea of using negative dependence to study randomized quasi-Monte Carlo point sets. For example, randomizations that induce negative dependence for lattice rules are presented in \cite{Wnuk19lat}. In \cite{Wnuk19Prob}, the authors study different concepts of negative dependence and, among other things, use them to derive probabilistic upper bounds on the discrepancy of the corresponding point sets.

In the present paper, we 
examine the randomized quasi-Monte Carlo sampling scheme $\tilde{P}_n$ obtained by scrambling a deterministic point set $P_n \subseteq [0,1)^s$. This class of sampling schemes includes scrambled digital $(t,m,s)-$nets in base $b$. We propose to measure the quality of these sampling schemes using values that we denote by $C_b(\BFk;P_n)$, $\BFk \in \N^s$, which arise in our study of the dependence structure of $\tilde{P}_n$. These quantities are related to how many distinct pairs of points lie in the same elementary $\BFk-$interval. They also play a key role in analyzing whether or not $\tilde{P}_n$ is    
negative lower orthant dependent (NLOD) and negative upper orthant dependent (NUOD). 
Indeed, we show these properties hold if and only if $C_b(\BFk;P_n) \le 1$ for all $\BFk \in \N^s$, a condition we refer to as being {\em completely quasi-equidistributed in base $b$}. In turn, this framework allows us to show that a scrambled digital $(t,m,s)-$net in base $b$ is NLOD/NUOD  
 if and only if $t=0$, for any dimension $s \ge 1$, thus generalizing the result from \cite{Lem17}. We also show that the first $n$ points of a $(0,s)-$sequence in base $b$ form a completely quasi-equidistributed point set (in base $b$).
 Through numerical examples we demonstrate that these  $C_b(\BFk;P_n)$ values are a powerful tool to compare the quality of different $(t,m,s)$-nets (in possibly different bases), and to enhance our understanding of how scrambling can improve the quality of deterministic point sets. 

This paper is organized as follows. In Section \ref{sec:prel} we review some background information on scrambled nets and dependence concepts, and prove a few key properties of scrambled nets that are relevant when studying their dependence structure. In Section \ref{sec:pdf} we obtain formulas for the joint probability density function (pdf) of pairs of distinct points in a scrambled point set. 
In Section \ref{sec:nlod} we show that a scrambled point set $\tilde{P}_n$ is NLOD/NUOD if and only if the underlying point set $P_n$ is completely quasi-equidistributed, a concept also defined in that section. In Section \ref{sec:qual} we discuss three possible avenues of exploration to leverage the insight provided by the $C_b(\BFk;P_n)$ values, and include numerical examples to illustrate these different ideas. 
Concluding comments and ideas for future work are presented in Section \ref{sec:Conc}, and technical proofs are included in the appendix.

\section{Preliminaries}
\label{sec:prel}

For ease of presentation, this section is divided into three subsections. The first one provides background on scrambled nets, the second one reviews dependence concepts, and the third one goes over tools that will be useful to analyze the joint pdf $\psi(\BFx,\BFy)$.

\subsection{Background on scrambled nets}
\label{sec:backNets}

We start by recalling key properties of scrambled nets.

A {\em digital net in base $b$} (for $b$ prime) \cite{DiPi10,rNIE92b} is a point set $P_n = \{\BFV_1,\ldots,\BFV_n\} \subseteq [0,1)^s$ with $n=b^m$ that is constructed via $s$  generating matrices $C_1,\ldots,C_s$ of size $m \times m$ with entries in $\mathbb{F}_b$, in the following way: for $0 \le i < b^m$ we write
$i = \sum_{r=0}^{m-1} i_r b^r$,
then the point $\BFV_i = (V_{i,1},\ldots,V_{i,s})$ is obtained as
$V_{i,\ell} = \sum_{r=1}^{m} V_{i,\ell,r}b^{-r}$, 
and
$V_{i,\ell,r} = \sum_{p=1}^{m} C_{\ell,r,p}i_{p-1}$,
where $C_{\ell,r,p}$ is the element on the $r$th row and $p$th column of $C_{\ell}$.

To assess the uniformity of the net,  the concept of $(k_1,\ldots,k_s)$-equidistribution is used. More precisely, we say that $P_n$ with $n=b^m$ is  {\em $(k_1,\ldots,k_s)$-equidistributed in base $b$} if every {\em elementary $(k_1,\ldots,k_s)-$interval} of the form
\begin{equation}
\label{eq:elemInt}
I_{\BFk}(\BFa) = \prod_{\ell=1}^s \left[\left. \frac{a_{\ell}}{b^{k_{\ell}}},\frac{a_{\ell}+1}{b^{k_{\ell}}} \right.\right)
\end{equation}
for $0 \le a_{\ell} < b^{k_{\ell}}$ contains exactly $b^{m-k_1-\ldots - k_s}$ points from $P_n$, assuming $m \ge k_1 + \ldots + k_s$. We say that a digital net in base $b$ has a {\em quality parameter $t$} if $P_n$ is $(k_1,\ldots,k_s)$-equidistributed for all $s$-dimensional vectors of non-negative integers $\BFk=(k_1,\ldots,k_s)$ such that $k_1 + \ldots + k_s \le m-t$. We then refer to $P_n$ as a digital $(t,m,s)$-net in base $b$. So the lower is $t$, the more uniform $P_n$ is \cite{rNIE92b}. 
For the remainder of this paper, when referring to $P_n$ as a $(t,m,s)$-net in base $b$, we assume $t$ is the smallest value for which this is true, i.e., we assume $P_n$ is not a $(t-1,m,s)$-net in base $b$. 

The construction proposed by Faure in \cite{rFAU82a} provides digital $(0,m,s)$-nets in prime bases $b \ge s$. The widely used Sobol' sequences \cite{rSOB67a} provide digital $(t,m,s)$-nets in base 2 with $t=0$ when $s=2$ and $t>0$ otherwise. Information on newer constructions can be found in \cite{qDIC13a,DiPi10}. Note that a {\em $(t,m,s)$-net in base $b$} is a point set $P_n$ with $n=b^m$ points such that the above equidistribution properties holds, but the point set may not necessarily have been constructed using generating matrices, i.e., using the digital method.

In this paper we are interested in randomized point sets $\tilde{P}_n = \{\BFU_1,\ldots,\BFU_n\}$ that are obtained by applying a {\em scrambling transformation}  in base $b$ to a deterministic point set $P_n = \{\BFV_1,\ldots,\BFV_n\}$. We can think of a scrambling transformation as a function ${\cal S}: [0,1]^s \times \Omega \rightarrow  [0,1]^s$ which applies a given random vector $\omega$ from a probability space $(\Omega,{\cal F}, {\cal P})$ to the base $b$ digits $V_{i,\ell,r}$ of each point $\BFV_i$ to get $\BFU_i = {\cal S}(\BFV_i,\omega)$.  

Generally speaking, the goal of a scrambling transformation is to create a randomized version of a point set $P_n$  that preserves the desirable properties of $P_n$ but allows for error estimation. It also usually refers to a process that either randomizes the generating matrices of the digital net, or applies random permutations to the base $b$ digits $V_{i,l,r}$ forming the points $\BFV_i$.
For instance, one way to scramble a digital net $P_n$ in base $b$ 
is to multiply from the left each generating matrix $C_{\ell}$ by a randomly chosen non-singular lower triangular (NLT) matrix $S_{\ell}$ (i.e., with entries on the diagonal uniformly chosen in $\{1,\ldots,b-1\}$, and entries below the diagonal uniformly chosen in $\{0,\ldots,b-1\}$, with the other entries set to 0), and then add a digital shift in base $b$ \cite{rMAT98c}. In this case, the random vector $\omega$ would correspond to the entries in $S_{\ell}, \ell=1,\ldots,s$ and the digital shift. This scrambling method is referred to as ``random linear scrambling'' in \cite{rMAT98c}, ``Owen's scrambling'' in \cite{vHON01a}, and as ``affine matrix scrambling'' in \cite{vOWE03a}, which is the term we adopt in this paper.

In this paper, we assume the scrambling transformation ${\cal S}$ is such that the  following two properties hold \cite{Hic96a,vHON01a,vOWE03a} and refer to such ${\cal S}$ as a {\em base $b-$digital scramble}. We also denote the obtained point set by ${}_b\tilde{P}_n$. 

Let 
$U_{i,\ell} = \sum_{r=1}^{\infty} U_{i,\ell,r} b^{-r}$, 
that is, $U_{i,\ell,r}$ represents the $r$th digit in the base $b$ expansion of the $\ell$th coordinate of the $i$th point $\BFU_i$. Then we must have:

\begin{enumerate}
\item Each $\BFU_i \sim U([0,1)^s)$; 
\hchristiane{In Hickernell's papers, this property is instead stated as saying that each digit is uniformly distributed over ${0,\ldots,b-1}$ and that the digits are independent across dimensions, i.e., we could replace by ``Each $U_{i,\ell,r}$ is uniformly distributed over $\{0,1,\ldots,b-1\}$, and for any two points $U_{i}$ and $U_{j}$, the pairs $(U_{i,1},U_{j,1}),\ldots,(U_{i,s},U_{j,s})$ are mutually independent''. Hickernell 1996 says that this implies each $\BFU_i \sim U([0,1]^s)$, which we could then state as a remark in case we used that assumption directly elsewhere.}
\item For two distinct points $\BFU_i={\cal S}(\BFV_i,\omega),\BFU_j={\cal S}(\BFV_j,\omega)$ and for each coordinate $\ell=1,\ldots,s$, if the two deterministic points $V_{i,\ell},V_{j,\ell}$ 
have the same first $r$ digits in base $b$ and differ on the $(r+1)$th digit, then (i) the scrambled points $(U_{i,\ell},U_{j,\ell})$ also have the same first $r$ digits in base $b$, and the pair $(U_{i,\ell,r+1},U_{j,\ell,r+1})$ is uniformly distributed over $\{(k_1,k_2), 0 \le k_1 \neq k_2 < b\}$; (ii)
the pairs $(U_{i,\ell,v},U_{j,\ell,v})$ for $v > r+1$ are mutually independent and uniformly distributed over 
$\{(k_1,k_2), 0 \le k_1, k_2 < b\}$.
\end{enumerate}

Note that in the description of the above two properties, a base $b-$digital scramble does not require $n=b^m$, and the base $b$ of the scrambling does not need to match the base in which a net has been constructed, hence the notation ${}_b\tilde{P}_n$. 

The affine matrix scrambling method described above---using NLT matrices $S_{\ell}$---can be shown to satisfy these two properties \cite{vHON01a,vOWE03a} (see also \cite{rMAT98c}, where a slightly weaker condition is used for 2.(ii)), as well as the nested uniform scrambling method proposed by Owen in \cite{vOWE95a}. We refer the reader to \cite{vOWE03a} for further information on scrambling methods for digital nets. For the remainder of this paper, whenever we refer to a {\em scrambled $(t,m,s)-$net in base $b$}  we are assuming it has been scrambled using a base $b-$digital scramble.  

\subsection{Dependence concepts}

Next, we introduce dependence concepts from \cite{Lem17} that will be used throughout this paper. 

\hchristiane{I reordered this so that it is easier to then motivate the def;n of NLOD and NUOD}

Consider a sampling scheme  $\tilde{P}_n = \{\BFU_1,\ldots,\BFU_n\}$ 
designed to construct an unbiased estimator of the form 
\[
\hat{\mu}_n = \frac{1}{n} \sum_{i=1}^n f(\BFU_i)
\]%
for 
\[
\mu(f)=\int_{[0,1)^s} f(\BFx)d\BFx,
\] 
where we assume each $\BFU_i$ is uniformly distributed over $[0,1)^s$ with a possible dependence structure between the $\BFU_i$'s. To assess this dependence, a key quantity of interest is
\begin{equation}
\label{eq:puvdims}
H(\BFx,\BFy;\tilde{P}_n) := \frac{2}{n(n-1)}\sum_{i =1}^{n-1}\sum_{j > i} P(\BFU_i \le \BFx,\BFU_j \le \BFy).
\end{equation}

We can think of $H(\BFx,\BFy;\tilde{P}_n)$ as the joint distribution function of a pair of (distinct) points $(\BFU_I,\BFU_J)$ randomly chosen in $\tilde{P}_n$. (Here, we use capital letters for the indices $I$ and $
J$ to make it clear the points are randomly selected.) 
%

Intuitively speaking, having negative dependence across the points of a sampling scheme $\tilde{P}_n$ is a desirable property because it implies the points are less likely to be clustered together, as they instead tend to repel each other, thus ensuring the sampling space is well covered by the points in $\tilde{P}_n$.

In this paper, negative dependence is assessed using the following concepts from \cite{Nelsen}: we say that a vector $\BFX = (X_1,\ldots,X_r)$ of  random variables is NLOD if
\[
P(X_1 \le x_1,\ldots,X_r \le x_r) \le \prod_{\ell=1}^r P(X_{\ell} \le x_{\ell}),
\]
and it is NUOD if
\[
P(X_1 > x_1,\ldots,X_r > x_r) \le \prod_{\ell=1}^r P(X_{\ell} > x_{\ell}).
\]

Note that when the dimension $r=2$, the NLOD and NUOD properties are equivalent 
but it is not necessarily  the case when $r \ge 3$.

One can think of the NLOD property as requiring that the probability that the $X_j$'s be all simultaneously small is no larger than if the $X_j$'s were independent; the NUOD property similarly requires that the probability that they be all simultaneously large is no larger than if they were independent.

If $H(\BFx,\BFy;\tilde{P}_n) \le \prod_{\ell=1}^s x_{\ell}y_{\ell}$ for all $0 \le x_{\ell},y_{\ell} \le 1$, $\ell=1,\ldots,s$, then we say $\tilde{P}_n$ is an {\em NLOD sampling scheme}.

We are also interested in the quantity
\hchristiane{Should we change $\ge$ for $>$ inside $P$? I used $\ge$ in my other paper and this is also more consistent with Lemma 4.16, where we include $(1-x,1-y)$ in integration domain. But to be consistent with above NUOD would be best to use $>$.}
\hjaspar{Lets use >}
\begin{equation}
\label{eq:puvdims_NUOD}
T(\BFx,\BFy;\tilde{P}_n) := \frac{2}{n(n-1)}\sum_{i =1}^{n-1}\sum_{j > i} P(\BFU_i > \BFx,\BFU_j > \BFy),
\end{equation}
and say that $\tilde{P}_n$ is an {\em NUOD sampling scheme} if $T(\BFx,\BFy;\tilde{P}_n)  \le \prod_{\ell=1}^s (1-x_{\ell})(1-y_{\ell})$ for all $0 \le x_{\ell},y_{\ell} \le 1$, $\ell=1,\ldots,s$.

In \cite{Lem17}, the quantity $T(\BFx,\BFy;\tilde{P}_n)$ arises in the analysis of  ${\rm Cov}(f(\BFU_I),f(\BFU_J))$, the covariance term that differentiates the variance 
of $\hat{\mu}_n$---when $\tilde{P}_n$ is a dependent sampling scheme---from that of a Monte Carlo estimator with the same number  of points $n$. More precisely, ${\rm Cov}(f(\BFU_I),f(\BFU_J))$ is such that
\[
{\rm Var}(\hat{\mu}_n) = \frac{\sigma^2}{n} + \frac{n-1}{n} {\rm Cov}(f(\BFU_I),f(\BFU_J)),
\]
where $\sigma^2 = {\rm Var}(f(\BFU))$. 
In the present work, rather than using the expression developed in \cite{Lem17} to write this covariance in terms of the survival function $T(\BFx,\BFy;\tilde{P}_n)$, we instead work with the direct representation
\begin{equation}
\label{eq:covaspdf}
\sigma_{I,J} := {\rm Cov}(f(\BFU_I),f(\BFU_J)) = \int_{[0,1]^{2s}} f(\BFx) f(\BFy) \psi(\BFx,\BFy)d\BFx d\BFy -
\int_{[0,1]^{2s}} f(\BFx) f(\BFy)d\BFx d\BFy.
\end{equation}
where $\psi(\BFx,\BFy)$ is the joint pdf of $(\BFU_I,\BFU_J)$ evaluated at $(\BFx,\BFy)$. (We set $\psi(\BFx,\BFy)=0$ if one of the coordinates of $\BFx$ or $\BFy$ is equal to 1.) In particular, this means we can also write
\hchristiane{Earlier when we define $H(\BFx,\BFy)$, we include $\BFx$ and $\BFy$ as values that can be taken by $\BFX$ and $\BFY$. Here we are exluding them. I still think we need to keep this def'n of $R(\BFx,\BFy)$ that excludes $\BFx$ and $\BFy$, and formally since we integrate wrt Lebesgue measure it doesn't matter whether or not we include the boundary. But I wonder if we need to settle this issue right from the start, i.e., say something that for the sake of convenience, we use the convention to integrate over half-open spaces in this paper?}
\hjaspar{I seem to remember that we decided on half open so that elementary intervals patition the unit square. Lets use half open intervals but remark that either would be fine because the difference is a set of measure 0.}
\begin{equation}
    \label{eq:HxyasInteg}
H(\BFx,\BFy;\tilde{P}_n) =  \int_{R(\BFx,\BFy)} \psi(\BFu,\BFv)d\BFu d\BFv,
\end{equation}
where $R(\BFx,\BFy) =\{(\BFu,\BFv) \in [0,1)^{2s}:u_j < x_j, v_j < y_j,j=1,\ldots,s\}.$ That is, $R(\BFx,\BFy) = [\mathbf{0},\BFx) \times [\mathbf{0},\BFy)$.  \hchristiane{added} \hjaspar{okay}Note that formally speaking, the definition of $H(\BFx,\BFy;\tilde{P}_n)$ given in \eqref{eq:puvdims} should lead to a closed  integration domain in \eqref{eq:HxyasInteg}. The reason why we instead integrate over a half-open interval is because it aligns better with the properties of $\psi(\BFx,\BFy)$, and is a convention we will follow throughout this paper. It is a valid approach because the boundary has measure 0, and thus the integral is unchanged whether we use a half-open interval or a closed one.
\hchristiane{Changed def'n, replacing $u_j \le x_j$ by $u_j<x_j$ and $[0,1)^{2s}$}

The joint pdf $\psi(\BFx,\BFy)$ corresponding to a base $b-$digitally scrambled point set  is the topic of Section \ref{sec:pdf}. The rest of this section develops tools to analyze this joint pdf.


\subsection{Tools to analyze the joint pdf $\psi(\BFx,\BFy)$}

\hchristiane{I removed ``finite''. I think the original concern was that $x$ or $y$ could be using a different expansion, i.e., replace a digit equal to $b-1$ by making it 0 and all the following ones equal to $b-1$. But what we say after ``i.e.'' rules out this possibility, so I think we're fine. The only thing we could do is add a remark in that direction. Thoughts?}
\hjaspar{Add a remark. I think its helpful to mention that gamma is well defined regardless of which expansion is used.}
\hchristiane{We can make it more precise that we are asking for finite base $b$ expansion but then we need $x,y \in [0,1)$. I want to check that this will not cause further problems.}
\hchristiane{I don't think we can assume expansion is finite (e.g., in light of Lemma 4.16), if so, we need to remove this assumption}
\hjaspar{$x$ has finite base $b$ expension only if the number is $x=n/b^k$}
\begin{defi}
        For $x,y\in[0,1)$, let $\gamma_b(x,y) \ge 0$ be the exact number of initial digits shared by $x$ and $y$ in their base $b$ expansion, i.e.\ the smallest number $i \ge 0$ such that
	\[
	\lfloor b^ix\rfloor=\lfloor b^iy\rfloor\quad\text{but}\quad\lfloor b^{i+1}x\rfloor \neq\lfloor b^{i+1}y\rfloor.
	\]
	\hchristiane{added following sentence and put two $\ge 0$ above. I don't think we need to explain that yes indeed we can have $\gamma_b(x,y)=0$.}
	If $x=y$ then we let $\gamma_b(x,y) = \infty$.
%
	For $\BFx,\BFy\in[0,1)^s$, we define 
	\[
	\boldsymbol{\gamma}^s_b(\BFx,\BFy) = (\gamma_b(x_1,y_1),\ldots,\gamma_b(x_s,y_s))
	\mbox{ and }
	\gamma_b(\BFx,\BFy) = \sum_{j=1}^s \gamma_b(x_j,y_j).
	\]	
	\hchristiane{added following sentence}
\end{defi}

Note that $\boldsymbol{\gamma}^s_b(\BFx,\BFy)$ denotes an $s$-dimensional vector while 	$ \gamma_b(\BFx,\BFy)$ is a scalar. \hchristiane{Added}\hjaspar{okay} Also, note that $\gamma_b(x,y)$ is well defined for any $x,y \in [0,1)$ even if  $x,y$ do not have a unique expansion in base $b$.

Given $\BFi,\BFk \in \mathbb{N}^s$, we say that $\BFk \le \BFi$ if $k_j \le i_j$ for all $j=1,\ldots,s$. (Note that in this paper, we assume that $\N$ includes 0.) 
We also denote the $\ell_1$-norm of a vector $\BFk$ by 
$|\BFk|=k_1 + \ldots + k_s$. 
For each $\BFk,\BFi \in (\mathbb{N} \cup \{\infty\})^s$ we define $C_{\BFk}^s, D_{\BFi}^s \subseteq [0,1)^{2s}$ to be the subsets
\begin{align*}
C_{\BFk}^s &= \{(\BFx,\BFy) \in [0,1)^{2s}: \BFk \le \boldsymbol{\gamma}_b^s(\BFx,\BFy)\} \qquad \mbox{ and} \\
D_{\BFi}^s &= \{(\BFx,\BFy) \in [0,1)^{2s}: \boldsymbol{\gamma}_b^s(\BFx,\BFy) = \BFi\}.
\end{align*}
Note that if $(\BFx,\BFy) \in D_{\BFi}^s$ for finite $\BFi$, we must have $\BFx \neq \BFy$.
In the special case $s=1$ we denote these sets by $C_k$ and $D_i$ respectively. It is clear that $C_{\BFk}^s = \cup_{\BFi \ge \BFk} D_{\BFi}^s$ and that the $D_{\BFi}^s$'s partition $[0,1)^{2s}$. \hchristiane{Note half open unit hypercube instead of closed}. One can easily verify that
\[
C_k = \bigcup_{a=0}^{b^k-1} \left[ \frac{a}{b^k},\frac{a+1}{b^k}\right)^2 \mbox{ and } C_{\BFk}^s = \prod_{j=1}^s C_{k_j}
\]
from which it follows that ${\rm Vol}(C_k) = b^{-k}$ and ${\rm Vol}(C_{\BFk}^s) = b^{-k}$. Finally, since $D_i = C_i \backslash C_{i+1}$ and because $D_{\BFi}^s = \prod_{j=1}^s D_{i_j}$ we have
\[
{\rm Vol}(D_{\BFi}^s) = \frac{(b-1)^s}{b^{s+|\BFi\,|}}
\]
because ${\rm Vol}(D_i) = (b-1)/b^{i+1}$.

\begin{remark}
\label{rem:ShortScrProp}
The newly introduced notation $\boldsymbol{\gamma}^s_b(\BFx,\BFy)$ and $D_{\BFi}^s$ give us a succinct way of describing 
the two properties of a base $b-$digital scramble mentioned in Section \ref{sec:backNets}. Indeed, these properties are equivalent to the following property: 
for any two scrambled points $\BFU_j,\BFU_l$ obtained from a base $b-$digital scrambling of $\BFV_j,\BFV_l$, it holds that $(\BFU_j,\BFU_l) \sim U(D_{\BFi}^s)$, where $\BFi = \boldsymbol{\gamma}^s_b(\BFV_j,\BFV_l)$. In turn, this implies that the joint pdf of a scrambled digital net in base $b$ is constant on each $D_{\BFi}^s$ (and is zero on those $D_{\BFi}^s$ for which $i = \infty$).
This latter property will be used in the analysis that follows and in the proof of Theorem \ref{FormOfJointPDF}.
\end{remark}

To prove the NLOD property one must show that
\[
\int_{R(\BFx,\BFy)}\psi(\BFu,\BFv) d\BFu d\BFv \leq \vol(R(\BFx,\BFy)) = \prod_{j=1}^s x_j y_j
\]
holds for all $\BFx,\BFy \in [0,1]^{2s}$. This integral may be written as \hchristiane{If $\BFx,\BFy$ fall on boundaries of $D_{\BFi}$ then I think we need to do a measure 0 argument to argue equality holds (or maybe not??)}
\[
\int_{R(\BFx,\BFy)}\psi(\BFu,\BFv) d\BFu d\BFv=\sum_{\BFi \in \mathbb{N}^s}^\infty V_{\BFi}^s(\BFx,\BFy)\psi_{\BFi},
\]
where
\begin{equation}
    \label{def:Vi}
V_{\BFi}^s(\BFx,\BFy) = \int_{R(\BFx,\BFy)} { 1}_{D_{\BFi}^s}(\BFu,\BFv)d\BFu d\BFv = \vol(R(\BFx,\BFy)) \cap D_{\BFi}^s)
\end{equation}

and $\psi_{\BFi}$ is the value of $\psi$ on $D_{\BFi}^s$. As before, in the special case $s=1$ we use the notation $V_i(x,y)$ or simply $V_i$  when $x$ and $y$ are fixed. We will use the fact that
\[
V_{\BFi}^s(\BFx,\BFy) = \int_{D_{i_1}} \ldots \int_{D_{i_s}} \prod_{j=1}^s {1}_{R(x_j,y_j)}(u_j,v_j) du_sdv_s \ldots du_1 dv_1 = \prod_{j=1}^s V_{i_j}(x_j,y_j)
\]
together with the following lemma (stating a result that appears in the proof of \cite[Proposition 7]{Lem17}) to simplify the calculation of $V_{\BFi}^s(\BFx,\BFy)$.

\begin{lemma}\label{FormulaForVi}
	Let $V_i=V_i(x,y)$ where $x,y\in[0,1)$ and $i \ge 0$. Then 
	\[
	V_i=\begin{cases}
	\frac{b-1}{b}\frac{\min(x,y)}{b^i}					&\text{if } \gamma_b(x,y)<i,\\
	xy-h_i(x+y-h_i-b^{-i})-\frac{\min(x,y)}{b^{i+1}}	&\text{if } \gamma_b(x,y)=i,\\
	h_{i+1}(x+y-h_{i+1}-b^{-i-1})-h_i(x+y-h_i-b^{-i})	&\text{if } \gamma_b(x,y)>i,\\
	\end{cases}
	\]
	where $h_i=\lfloor b^i\min(x,y)\rfloor b^{-i}$.
\end{lemma}

We handle the special case where at least one of $x$ or $y$ equals 1 in the following lemma.

\begin{lemma}
\label{lem:FormViOne}
For any $x \in [0,1]$ and $i \ge 0$,  $V_i(x,1)=V_i(1,x) = xV_i(1,1)$, where $V_i(1,1) = {\rm Vol}(D_i) = (b-1)/b^{i+1}$.
\end{lemma}

\begin{proof}
First, it is clear from the definition of $V_i(1,1)$ that it equals ${\rm Vol}(D_i)$. Next, for $u \in [0,1)$ we define
\[
D_i(u) = \{ y \in [0,1): \gamma_b(u,y)=i\}.
\]
Elements of $D_i(u)$ have the same first $i$ digits as $u$, differ on the $(i+1)$st digit, with the remaining digits being free. Hence the length of $D_i(u)$ is $(b-1)/b^{i+1}$. Next, we evaluate
\begin{align*}
V_i(x,1) &= \int_{R(x,1)} {1}_{D_i}(u,v)dudv = \int_0^x \int_0^1{1}_{D_i}(u,v)dudv = \int_0^x \int_0^1 {1}_{D_i(u)}(v) dvdu\\
&=\int_0^x \frac{b-1}{b^{i+1}} du = x \frac{b-1}{b^{i+1}}.
\end{align*}
Clearly, a similar argument can be made to compute $V_i(1,x)$.
\end{proof}

We only need the above formulas to prove the following technical lemma, which gives us a critical relation between $V_i$ and $V_{i+1}$.  While its proof (found in the appendix) is rather tedious, it is not hard, we simply use Lemmas \ref{FormulaForVi} and \ref{lem:FormViOne}, and carefully work through the cases.

\begin{lemma}\label{TechnicalLemma}
Let $x,y\in[0,1]$ be given and $V_i=V_i(x,y)$ be defined as in \eqref{def:Vi}. Then  $bV_{i}-V_{i-1}\geq0$ for all $i\geq1$. 
\end{lemma}

\section{The joint pdf of scrambled point sets}
\label{sec:pdf}

We first introduce some notation that will be helpful for important counting arguments that are needed to derive the joint pdf of scrambled point sets. 

\hchristiane{Made this more general; not sure if we should use $P_n$ instead of $\tilde{P}_n$. Also, should probably put subscripts $b$ for $m(\BFk,)$ and $n(\BFi,)$ Below I refer to a digital net that is deterministic or scrambled. We may also want to clarify this. I.e., we could say from the start that a digital net may be deterministic or scrambled, and that when we need to specify which of the two it is, we will.}
\begin{defi}\label{def:mknipn}
Let $P_n =\{\BFU_1,\ldots,\BFU_n\}$ be a point set in $[0,1)^s$ and $b \ge 2$ be an integer. 

\begin{enumerate}
\item Let $m_b(\BFk;P_n,\BFU_l)$ be the number of points $\BFU_j \in  P_n$ with $j \neq l$ satisfying $\boldsymbol{\gamma}_b^s(\BFU_l,\BFU_j) \ge \BFk$. If these numbers are the same for all $\BFU_l$ then we write $m_b(\BFk;P_n,\BFU_l) =
m_b(\BFk;P_n)$.
\item Let $M_b(\BFk;P_n)$ be the number of ordered pairs of distinct points $(\BFU_l,\BFU_j)$ in $P_n$ such that $\boldsymbol{\gamma}_b^s(\BFU_l,\BFU_j) \ge \BFk$. 
\item Let $n_b(\BFi;P_n,\BFU_l)$  be the number of points $\BFU_j \in P_n$ satisfying $\boldsymbol{\gamma}_b^s(\BFU_l,\BFU_j) = \BFi$. If these numbers are the same for all $\BFU_l$ we write $n_b(\BFk;P_n,\BFU_l) = n_b(\BFk;P_n)$.
\item Let $N_b(\BFi;P_n)$ be the number of ordered pairs $(\BFU_l,\BFU_j)$ in $P_n$ such that $\boldsymbol{\gamma}_b^s(\BFU_l,\BFU_j) = \BFi$.
\item In the special  case where  $P_n$ is a digital $(t,m,s)$-net (deterministic or scrambled), we let $r(\BFk)$ be the rank of the matrix formed by the first $k_j$ rows of the generating matrices $C_j$, $j=1,\ldots,s$. (If $\BFk= \mathbf{0}$ then we set $r(\BFk) = 0$.)
\end{enumerate}
\end{defi}

\begin{remark}
\label{rem:MbNb}
A few observations are in order:
\begin{enumerate}
\item \hchristiane{Changed} The integer $b$ used in the above definitions does not need to be equal to the base used to construct the point set $P_n$. This is the reason why we highlight it as a subscript in the above definitions. 
\hchristiane{Following sentence added based on Ref 2}
 \hjaspar{Did you want to change this as we discussed?}
\item It is well known (see e.g., \cite{rNIE92b}) that $r(\BFk) = |\BFk|$ if and only if $P_n$ is $\BFk$-equidistributed.
\hchristiane{added}
\item The quantity $m_b(\BFk;P_n,\BFU_l)$ is also equal to  the number of points in $P_n$ (other than $\BFU_l$) that are in the same $\BFk-$elementary interval as $\BFU_l$. Similarly, $M_b(\BFk;P_n)$ is the number of ordered pairs of distinct points from $P_n$ that lie in the same $\BFk-$elementary intervals.
\item The quantity $N_b(\BFi;P_n)$ corresponds to the number of (ordered) pairs of distinct points that are in $D_{\BFi}^s$.
\item Clearly the following relationships hold:
\[
M_b(\BFk;P_n) = \sum_{l=1}^n m_b(\BFk;P_n,\BFU_l) \mbox{ and } N_b(\BFi;P_n) = \sum_{l=1}^n n_b(\BFi;P_n,\BFU_l).
\]
\item As can be inferred from Remark \ref{rem:ShortScrProp}, if $(\BFU_j,\BFU_l)$ are the two points obtained  after applying a base $b-$digital scramble to $(\BFV_j,\BFV_l)$, then $\boldsymbol{\gamma}^s_b(\BFV_j,\BFV_l) = \boldsymbol{\gamma}^s_b(\BFU_j,\BFU_l)$. Since all the counting numbers introduced in Definition \ref{def:mknipn} are entirely determined by the function $\boldsymbol{\gamma}^s_b(\cdot,\cdot)$, it means these numbers are the same for $P_n$ and ${}_b\tilde{P}_n$. 
\end{enumerate}
\end{remark}

In what follows we will make use of the following relations between the above quantities.

\begin{prop}
\label{prop:nbInTermsMb}
The quantities introduced in Definition \ref{def:mknipn} satisfy:
\hchristiane{Still not 100\% sure but I think we need the + 1 in there}
\begin{align*}
n_b(\BFk;P_n,\BFU_l) &= \sum_{\BFe \in \{0,1\}^s} (-1)^{|\BFe|} m_b(\BFk+\BFe;P_n,\BFU_l) \mbox{ and }\\
N_b(\BFk;P_n,\BFU_l) &= \sum_{\BFe \in \{0,1\}^s} (-1)^{|\BFe|} M_b(\BFk+\BFe;P_n,\BFU_l).
\end{align*}
\end{prop}

\begin{proof}
Fix $\BFU_l \in P_n$ and for each $\BFk \in \mathbb{N}^s$ let $I_{\BFk}^l$ denote the elementary $\BFk-$interval that contains $\BFU_l$. Since a point $\BFU_j \in P_n \cap I_{\BFk}^l$ satisfies $\boldsymbol{\gamma}_b^s(\BFU_l,\BFU_j)= \BFk$ if and only if $\BFU_j$ is not in any $I_{\BFk+\BFe}^l$, where $\BFe \in \{0,1\}^s$ with $|\BFe|=1$, it holds  that $n_b(\BFk;P_n,\BFU_l)$ counts the number of points from $P_n$ that are in the set 
\[
I_{\BFk}^l \backslash \left(\bigcup_{\BFe \in \{0,1\}^s,\, {|\BFk|}={|\BFe|}} I_{\BFk+\BFe}^l\right).
\] 
\hchristiane{I changed $n$ for $r$ to not confuse with nb of points}
Note that $\BFU_l$ is not in the above set.
To apply the Principle of Inclusion-Exclusion, we observe that the intersection of any $r$ distinct  elementary intervals in the above union is an elementary interval of the form $I_{\BFk+\BFe}^l$, where $\BFe \in \{0,1\}^s$ with ${|\BFe|}=r$, and that $m_b(\BFk+\BFe;P_n,\BFU_l)+1$ counts the number of points in $I_{\BFk+\BFe}^l$. Thus by the Principle of Inclusion-Exclusion we have
\begin{align*}
n_b(\BFk;P_n,\BFU_l) &= \sum_{r=0}^s \sum_{\substack{\BFe \in \{0,1\}^s \\ {|\BFe|}=r}} (-1)^{r} (m_b(\BFk+\BFe;\tilde{P}_n,\BFU_l)+1)\\
&= \sum_{\BFe \in \{0,1\}^s} (-1)^{|\BFe|} (m_b(\BFk+\BFe;P_n,\BFU_l)+1) = 
\sum_{\BFe \in \{0,1\}^s} (-1)^{|\BFe|} m_b(\BFk+\BFe;P_n,\BFU_l),
\end{align*}
where the last equality follows from the fact that $\sum_{\BFe \in \{0,1\}^s} (-1)^{|\BFe|} = 0$.
The expression for $N_b(\BFk;P_n)$ follows from Remark \ref{rem:MbNb}.4.
\end{proof}

The next result provides expressions for the counting numbers $m_b(\BFk;P_n)$ and $n_b(\BFi;P_n)$ in the special case of scrambled digital $(t,m,s)$-nets and scrambled nets with $t=0$.  

\hchristiane{Making note that for this result, we assume $b$ is also the constructing base}
\begin{lemma}\label{lem3p2}
Let ${}_b\tilde{P}_n$ be a scrambled digital $(t,m,s)$-net in base $b$. Then 
\begin{align*}
&\mbox{(i) } m_b(\BFk;{}_b\tilde{P}_n) = b^{m-r(\BFk)}-1 \\
&\mbox{(ii) } n_b(\BFi;{}_b\tilde{P}_n) = \sum_{\BFe \in \{0,1\}^s} (-1)^{|\BFe|} b^{m-r(\BFi+\BFe)} \mbox{  and}\\
&\mbox{(iii) }  N_b(\BFi;{}_b\tilde{P}_n) = b^m n_b(\BFi;{}_b\tilde{P}_n).
\end{align*}
For scrambled $(0,m,s)-$nets in base $b$, we have
\begin{align*}
&\mbox{(iv) } m_b(\BFk;{}_b\tilde{P}_n) = \max(b^{m-|\BFk|}-1,0) \mbox{ and }\\
&\mbox{(v) } n_b(\BFi;{}_b\tilde{P}_n) =  \sum_{k=0}^s (-1)^k \binom{s}{k} \max(b^{m-{|\BFi\,|}-k},1).
\end{align*}
\end{lemma}

\hchristiane{I added some details to handle the fact that we are working with the scrambled points rather than the deterministic ones, as per referee request}
\begin{proof}
Consider the partition  of $[0,1)^s$ \hchristiane{half open instead of closed} into elementary $\BFk-$intervals of the form $I_{\BFk}(\BFa)$ with $0 \le a_j < b^{k_j}$, $j=1,\ldots,s$ (defined in \eqref{eq:elemInt}) and let $\BFa_{\ell}$ denote the $|\BFk|$-dimensional integer vector corresponding to the elementary interval $I_{\BFk}(\BFa_{\ell})$ (which we denote $I_{\BFk}^{\ell}$ for short) from this partition in which $\BFU_{\ell}$ lies. Next, we observe that from the properties of the scrambling described in Section \ref{sec:backNets}, there exists a bijection ${\cal S}_{\BFk}:\mathbb{F}_b^k \rightarrow \mathbb{F}_b^k$  such that if the deterministic point $\BFV_{\ell} \in I_{\BFk}(\BFa)$ then its scrambled version $\BFU_{\ell} \in I_{\BFk}({\cal S}_{\BFk}(\BFa)) = I_{\BFk}(\BFa')$ for some $0 \le \BFa' < b^{|\BFk|}$. Indeed, ${\cal S}_{\BFk}$ is injective because if $\BFV_{\ell} \in I_{\BFk}(\BFa)$ and $\BFV_j \in I_{\BFk}(\BFa')$ with $\BFa \neq \BFa'$, then $\BFU_{\ell}$ and $\BFU_j$ cannot be in the same elementary intervals (otherwise Property  2 of the scrambling would not hold); it is surjective because for each $\BFa' \in \mathbb{F}_b^k$ there has to be an $\BFa \in \mathbb{F}_b^k$ such that ${\cal S}_k(\BFa) = \BFa'$ otherwise Property 1 (uniformity) of the scrambling would not hold. Furthermore, this $\BFa$ is well defined since if $\BFU_{\ell},\BFU_j \in I_{\BFk}(\BFa')$ then $\BFV_{\ell},\BFV_j \in I_{\BFk}(\BFa)$. Then let $T_{\BFk}$ be the linear transformation from $\mathbb{F}_b^m$ to $\mathbb{F}_b^k$ determined by the $|\BFk|$ rows formed by the  first $k_j$ rows of $C_j$, for $j=1,\ldots,s$. Given that $\BFU_{\ell}$ lies in $I_{\BFk}^{\ell}$, it means $\BFa_{\ell}$ is in the image of   ${\cal S}_{\BFk} \circ T_{\BFk}$.  Note that the dimension of the null space of $T_{\BFk}$ is $m-r(\BFk)$, and ${\cal S}_{\BFk}$ is a bijection. Thus the number of points in $I_{\BFk}^{\ell}$ (including $\BFU_{\ell}$) is given by $b^{m-r(\BFk)}$, which corresponds to $m_b(\BFk;{}_b\tilde{P}_n)+1$. This shows (i). \hchristiane{I don't think we need the max because $r(\BFk) \le \min(m,k)$.}
    
To obtain (ii) 
we use Proposition \ref{prop:nbInTermsMb}, whose proof shows that 
\[
n_b(\BFi;{}_b\tilde{P}_n) 
= \sum_{\BFe \in \{0,1\}^s} (-1)^{|\BFe|} (m_b(\BFi+\BFe;{}_b\tilde{P}_n,\BFU_l)+1) = 
\sum_{\BFe \in \{0,1\}^s} (-1)^{|\BFe|} m_b(\BFi+\BFe;{}_b\tilde{P}_n,\BFU_l).
\]
To obtain (iii) we use  Remark \ref{rem:MbNb} (item 4).

For a general scrambled $(0,m,s)-$net in base $b$ (not necessarily constructed using the digital method), by definition we have $m_b(\BFk;{}_b\tilde{P}_n)=b^{m-|\BFk|}-1$ for $|\BFk| \le m$ and is 0 otherwise. This gives (iv). 
To get (v), we use the fact that there are exactly $\binom{s}{r}$ vectors $\BFk \in \{0,1\}^s$ with $|\BFk|=r$, together with the expression for $m_b(\BFk;{}_b\tilde{P}_n,\BFU_l)$, to get
\begin{align*}
n_b(\BFi;{}_b\tilde{P}_n) &= \sum_{r=0}^s \sum_{\substack{\BFk \in \{0,1\}^s\\ {|\BFk|}=r}} (-1)^{r} \max (b^{m-{|\BFi |}-r},1)
= \sum_{k=0}^s (-1)^k \binom{s}{k} \max(b^{m-{|\BFi\,|}-k},1). 
\end{align*}
\end{proof}

\begin{remark}
\label{rem:NoRepeatCoord}
\hchristiane{Made this more general. Or should we say that we assume $|P_n^j|=n$ for each $j$ (as we do in one of the lemmas)?. Could then remove this assumption in every statement of subsequent results.}
In the remainder of this paper, we assume we are working with point sets $P_n$ such that the $j$th coordinate of the points are all distinct, for $j=1,\ldots,s$. Equivalently, this means  we assume every pair of distinct points has a bounded number of common digits, i.e., we assume $\gamma_b(\BFV_i,\BFV_j) < \infty$ for all $i \neq j$.
The properties of a base $b-$digital scramble imply that  we also have $\gamma_b(\BFU_i,\BFU_j) < \infty$ for 
${}_b\tilde{P}_n =\{\BFU_1,\ldots,\BFU_n\}$, and that  the $j$th coordinate of the points from ${}_b\tilde{P}_n$ are all distinct, for $j=1,\ldots,s$. 
Doing so avoids the case where we have two points with equal coordinates in one or more dimension, which would in turn give a non-zero probability measure to a set of zero Lebesgue measure in $[0,1)^{2s}$, namely  
on $D_{\BFi}^s$ with $|\BFi\,|=\infty$.
\hchristiane{added: ref 2 said Sobol but also true for Faure}
We note that the first $n$ points of both Sobol' and Faure sequences have this property.
\end{remark}

\hchristiane{Should we also index the pdf with $b$?}
\begin{thm}\label{FormOfJointPDF}
	Let ${}_b\tilde P_n=\{\BFU_1,\ldots,\BFU_n\}$ be a point set obtained by applying a base $b-$digital scramble to 
	$P_n = \{\BFV_1,\ldots,\BFV_n\}$. Assume ${P}_n$ 
	is such that  the $j$th coordinate of the $n$ points are all distinct,  for each $j=1,\ldots,s$. 
    For $\BFx,\BFy \in [0,1)^s$ such that $\boldsymbol{\gamma}_b^s(\BFx,\BFy)=\BFi$ and $\gamma_b(\BFx,\BFy)=|\BFi\|$, the joint pdf $\psi(\BFx,\BFy)$ of two distinct points $(\BFU_I,\BFU_J)$ randomly chosen from ${}_b\tilde P_n$ is given by
	\[
	\psi(\BFx,\BFy)=\begin{cases}
	\frac{N_b(\BFi;P_n)}{n(n-1)}\frac{b^{s+{|\BFi\,|}}}{(b-1)^s}	&\text{if } {|\BFi\,|} < \infty,\\
	0	&\text{if } {|\BFi\,|}=\infty.
	\end{cases}
	\]
	
	In the special case where ${P}_n$ is a digital net in base $b$ with $n=b^m$, \hchristiane{removed $(t,m,s)$-net because Ref 2 was asking ``where is $t$?'' in the formula for the pdf} \hjaspar{this formula actually only holds for digitally constructed scrambled net. It follows because elementary intervals of the same type will have the same number of points. Thus we only need count the points in one elementary interval and we dont need the first $n$ in the denominator.}\hchristiane{Now says $P_n$ is a digital net}
	the joint pdf becomes
	\[
	\psi(\BFx,\BFy)=\begin{cases}
	\frac{n_b(\BFi;P_n)}{(b^m-1)}\frac{b^{s+{|\BFi\,|}}}{(b-1)^s}	&\text{if } {|\BFi\,|} < \infty,\\
	0	&\text{if } {|\BFi\,|}=\infty.
	\end{cases}
	\]	
\end{thm}

\begin{proof}
As explained in Remark \ref{rem:ShortScrProp}, from the properties of a base $b-$digital scramble, the joint pdf 
$\psi(\BFx,\BFy)$  is constant on $D_{\BFi}^s$. The value $\psi_{\BFi}$ of $\psi(\BFx,\BFy)$ on $D_{\BFi}^s$ (${|\BFi\,|} \neq \infty$) can be found by observing that the integral of $\psi(\BFx,\BFy)$ over $D_{\BFi}^s$ is equal to the probability that a random pair of distinct points from ${}_b\tilde{P}_n$ lie in 
 $D_{\BFi}^s$, i.e., 
\begin{equation}
\label{eq:SolvePDF}
\frac{N_b({\BFi};{}_b\tilde{P}_n)}{n(n-1)} = \int_{D_{\BFi}^s} \psi(\BFx,\BFy) d\BFx d\BFy .
\end{equation}
Since the right-hand side is also equal to $\psi_{\BFi}\vol(D_{\BFi}^s) 
= \psi_{\BFi} \frac{(b-1)^s}{b^{s+{|\BFi\,|}}}$, we then simply solve for $ \psi_{\BFi}$ using \eqref{eq:SolvePDF}. 

For a scrambled $(t,m,s)-$net in base $b$, from Lemma \ref{lem3p2} we get $N_b(\BFi;{}_b\tilde{P}_n) = b^m n_b(\BFi;{}_b\tilde{P}_n)$.

When ${|\BFi\,|}=\infty$, our assumption that the one-dimensional projections of $P_n$ have $n$ distinct points implies 
there cannot be two distinct points in $D_{\BFi}$ so the joint pdf must be 0 in this case.

Finally, based on Remark \ref{rem:MbNb}(v) we observe that $N_b({\BFi};{}_b\tilde{P}_n) = N_b({\BFi}; P_n)$ and $n_b(\BFi; {}_b\tilde{P}_n)=n_b(\BFi;P_n)$.

\hchristiane{Point out that $N_b(\infty;P_n)=1$ but ok because formula for pdf with $N_b(\BFi)$ only works for finite $i$ (since only works when $N_b()$ ends up counting distinct points? Might be worth pointing this out}
\end{proof}

\begin{remark}\label{rem:EndSec3}
A few observations are in order:
\begin{enumerate}
\item \hchristiane{Changed $N_b({\BFi};{}_b\tilde{P}_n)$ for $N_b({\BFi};{P}_n)$ etc} The joint pdf is a simple function because the assumption discussed in Remark \ref{rem:NoRepeatCoord} implies that if ${|\BFi\,|}$ is large enough, then $N_b({\BFi};P_n)$ becomes 0 and thus  $\psi(\BFx,\BFy)=0$ for $\BFx,\BFy$ with  $\gamma_b^s(\BFx,\BFy)=\BFi$.
\hchristiane{I think this remark was introduced when our result was restricted to scrambled $(0,m,s)$-net: fix}
\item We can see from Lemma \ref{lem3p2} and Theorem \ref{FormOfJointPDF} that the joint pdf of a scrambled $(0,m,s)$-net depends only on  the sum $\sum_{j=1}^s \gamma_b(x_j,y_j)$ and not on the vector $\boldsymbol{\gamma}_b^s(\BFx,\BFy)$, since in that case $n_b(\BFi;P_n)$ depends only on ${|\BFi\,|}$.
\item Using Lemma \ref{FormulaForVi} and its preceding discussion along with the formulas in this section, one can compute $H(\BFx,\BFy;{}_b\tilde{P}_n)$ exactly.
\item The expression for the joint pdf given in Theorem \ref{FormOfJointPDF} holds for any base$-b$ digital scramble. From this observation we obtain the following corollary, which can be inferred from a discussion in  \cite{vHON01a} (immediately before the statement of their Theorem 2.1), but does not seem to appear explicitly as a result anywhere in the literature.
\end{enumerate}
\end{remark}

\begin{cor}
Let $P_n$ be a digital net in base $b$  such that each one-dimensional projection is a $(0,m,1)-$net in base $b$. 
Let ${}_b\tilde{P}_{n,1}$ be the point set obtained after applying a nested uniform scramble in base $b$ to $P_n$ and let ${}_b\tilde{P}_{n,2}$ be the point set obtained after applying an affine matrix scramble  to $P_n$. Let $\hat{\mu}_{n,i}$ be the estimator for $\mu(f)$ corresponding to ${}_b\tilde{P}_{n,i}$, $i=1,2$. Then ${\rm Var}(\hat{\mu}_{n,1}) = {\rm Var}(\hat{\mu}_{n,2})$.
\end{cor}

\begin{proof}
Since both types of scrambling satisfy the properties of a base$-b$ digital scramble (see \cite{Hic96a,vHON01a} and \cite{vOWE03a}), the joint pdf of two distinct points $(\BFU_I,\BFU_J)$ randomly chosen from ${}_b\tilde{P}_{n,1}$ is the same as that for two distinct points $(\BFU_I,\BFU_J)$ randomly chosen from ${}_b\tilde{P}_{n,2}$. Hence using \eqref{eq:covaspdf}, we get that the estimators $\hat{\mu}_{n,1}$ and $\hat{\mu}_{n,2}$ corresponding to 
${}_b\tilde{P}_{n,1}$ and ${}_b\tilde{P}_{n,2}$ have the same variance.
\end{proof}

\section{Dependence structure of scrambled point sets}
\label{sec:nlod}

By the end of this section, we will have shown that the only scrambled digital $(t,m,s)$-nets that are NLOD/NUOD are those for which $t=0$. 
We will arrive to this result by using the properties of the joint pdf $\psi(\BFu,\BFv)$ for scrambled nets ${}_b\tilde{P}_n$, which was studied in the previous section. 
More precisely, we will develop an inequality of the form
\hchristiane{I don't know why we called this $G$ since we already used the notation $H$ for the same quantity in the introduction, so I'm changing this.}
\begin{equation}
\label{eq:GxywithCb}
H(\BFx,\BFy; {}_b\tilde{P}_n)=
\int_{R(\BFx,\BFy)}\psi(\BFu,\BFv) d\BFu d\BFv =
 \sum_{\BFk \in \mathbb{N}^s} t_{\BFk} C_b(\BFk;P_n) 
\leq 
C_b \vol(R(\BFx,\BFy))
\end{equation}
that holds for all $\BFx,\BFy\in[0,1]^s$, and in which the $t_{\BFk}$'s are non-negative coefficients determined by $\BFx,\BFy$ and such that 
$\sum_{\BFk \in \mathbb{N}^s} t_{\BFk} = \vol(R(\BFx,\BFy))$, while the $C_b(\BFk;P_n)$ values are determined by $P_n$, and play a key role in the analysis of the dependence structure of ${}_b\tilde{P}_n$, as mentioned in the introduction. We will also  provide an exact expression for the minimum constant $C_b$ satisfying this inequality, which turns out to be the maximum value of the $C_b(\BFk;P_n)$ values. Hence one can view the summation in \eqref{eq:GxywithCb} as a decomposition of $H(\BFx,\BFy; {}_b\tilde{P}_n)$ into a sum of products of two terms, with one term--$t_{\BFk}$--solely depending on the region $R(\BFx,\BFy)$ being considered, and the other--$C_b(\BFk;P_n)$--measuring the quality of $P_n$.  This is reminiscent of other fundamental results for quasi-Monte Carlo integration, where an error bound is given as a product of the form $V(f)D(P_n)$, where $V(f)$ measures the variation of $f$ while $D(P_n)$ measures the discrepancy of $P_n$ (see, e.g., \cite{DiPi10,rNIE92b}). 
 
\subsection{The $C_b(\BFk;P_n)$ values }

\begin{defi}
\label{def:Cbk}
Let $P_n$ be a set of $n$ points in $[0,1)^s$ and $b \ge 2$ be an integer. Let $C_b(\BFk;P_n)$ be defined as
\[
C_b(\BFk;P_n) = \frac{b^{|\BFk|} M_b(\BFk;P_n)}{n(n-1)}.
\]
\end{defi}

It is easy to see that $C_b(\BFk;P_n) = 1$ when $\BFk=\mathbf{0}$.  The following result also holds:
 
\begin{lemma}
If $P_n$ is a digital $(t,m,s)$-net in base $b$ whose one-dimensional projections are $(0,m,1)$-nets, then
\[
C_b(\BFk;P_n) = \frac{b^{|\BFk|}(b^{m-r(\BFk)} -1)}{b^m-1}.
\]
If $P_n$ is a $(0,m,s)$-net in base $b$, then
\[
C_b(\BFk;P_n)= \frac{b^{|\BFk|}(\max(b^{m-{|\BFk|}}-1,0))}{b^m-1}.
\]
\end{lemma}

\begin{proof}
From the definition of $M_b(\BFk;P_n)$, for $(t,m,s)-$nets in base $b$, it holds that $M_b(\BFk;P_n)= b^m m_b (\BFk;P_n)$, and from Lemma \ref{lem3p2}, for digital nets $m_b(\BFk;\tilde{P}_n) = b^{m-r(\BFk)}-1$ holds, from which we obtain the desired formula.
For a $(0,m,s)$-net in base $b$, we have $M_b(\BFk;P_n) = b^{m}(\max(b^{m-{|\BFk|}},1)-1)$, which after simplification yields $C_b(\BFk;P_n) = b^{|\BFk|} (\max(b^{m-{|\BFk|}},1)-1)/(b^m-1)$ for $k \le m$.
\end{proof}

As mentioned above, the values $C_b(\BFk;P_n)$ play a key role in our analysis of the joint pdf of scrambled point sets. They also have a connection with the concept of $\BFk-$equidistribution, as shown in the following lemma.

\hchristiane{Made this into a lemma because the proof, although simple, seemed a bit too long to just be mentioned in the text.}
\hjaspar{Looks good}
\hchristiane{Ref 1 is not convinced of "only if" part: we should either clarify that this iff is for digital nets, in which case $M_b(\BFk;P_n)=n(b^{m-k}-1)$ iff $r(\BFk)=k$, which only happens if $P_n$ is $\BFk$-equidistributed. If we want this to be more general, i.e., for any $(t,m,s)$-net, then we use the fact that if $P_n$ is not $\BFk$-equidistributed, this means some $\BFk$-intervals have more or less than $b^{m-k}$ points, and therefore (obvious?) the number of pairs cannot be $n(b^{m-k}-1)$.}
\begin{lemma} 
\label{lem:formCkiffequid}
Let $\BFk \in \mathbb{N}^s$ be such that $|\BFk| \le m$. A set $P_n \subseteq [0,1)^s$ with $n=b^m$ points is $\BFk-$equidistributed if and only if $C_b(\BFk;P_n)= b^{|\BFk|}(b^{m-|\BFk|} -1)/(b^m-1)$.
\end{lemma}

\begin{proof}
If $P_n$ is $\BFk-$equidistributed, then each elementary interval of the form $I_{\BFk}(\BFa)$ given in \eqref{eq:elemInt} has $b^{m-|\BFk|}$ points, and thus $M_b(\BFk;P_n)=b^{|\BFk|}(b^{m-|\BFk|}(b^{m-|\BFk|}-1)) = b^m (b^{m-|\BFk|}-1)$.
If $P_n$ is not $\BFk-$equidistributed, then it means some elementary intervals have more than $b^{m-|\BFk|}$ points and some have less. We will show this implies $M_b(\BFk;P_n)\neq n (b^{m-|\BFk|}-1)$. To do so, we let $x_i$ be such that the $i$th elementary interval has $b^{m-|\BFk|}+x_i$ points, for $i=1,\ldots,N$, where $N=b^{|\BFk|}$. Therefore $\sum_{i=1}^{N}x_i=0$ and some $x_i$'s are not 0. We also note that the number of distinct ordered pairs in an elementary interval with $b^{m-|\BFk|}+x_i$ points is given by
\[
(b^{m-|\BFk|}+x_i)(b^{m-|\BFk|}+x_i-1) = b^{m-|\BFk|}(b^{m-|\BFk|}-1) + 2x_i b^{m-|\BFk|} + x_i(x_i-1).
\]
Therefore
\begin{align*}
M_b(\BFk;P_n) &= n(b^{m-|\BFk|}-1)+2b^{m-|\BFk|}\sum_{i=1}^{N} x_i
+ \sum_{i=1}^{N} x_i (x_i-1) \\
&=n(b^{m-|\BFk|}-1)+ \sum_{i=1}^{N} x_i^2
\neq n(b^{m-|\BFk|}-1)
\end{align*}
because we have assumed not all $x_i$'s are equal to 0.
\end{proof}

We also note that the value $C_b(\BFk;P_n)$ can be computed for any point set $P_n$ and base $b \ge 2$, and leads us to the introduction of the following new concept.

\begin{defi}
\label{DefCqe}
Let $P_n$ be a point set of size $n$ in $[0,1)^s$ and $b \ge 2$ be a base. Let $\BFk = (k_1,\ldots,k_s) \in \N^s$. Then we say $P_n$ is {\em $\BFk-$quasi-equidistributed in base $b$} if $C_b(\BFk;P_n) \le 1$. If $C_b(\BFk;P_n) \le 1$ for all $\BFk \in \N^s$ then we say $P_n$ is {\em completely quasi-equidistributed (c.q.e) in base $b$}.
\end{defi}

Note that there are only finitely many values of $\BFk \in N^s$ for which we need to compute $C_b(\BFk;P_n)$ in order to verify if $P_n$ is c.q.e. Indeed, the condition that the $j$th coordinate of the points are all distinct for $j=1,\ldots,s$ 
ensures there exists an $M$ such that any $\BFk-$interval with $k>M$ has at most one point in it, which means $M_b(\BFk;P_n)=C_b(\BFk;P_n)=0$ for these $\BFk$'s. 

It is clear that a $(0,m,s)-$net in base $b$ is c.q.e (in base $b$), since by definition it is $\BFk-$equidistributed  for all $\BFk$ such that $|\BFk| \le m$. Other examples of point sets $P_n$ that are c.q.e.\ are given in the next proposition.

\begin{prop}
\label{prop:0scqe}
The first $n$ points of a $(0,s)-$sequence in base $b$ is a c.q.e.\ point set in base $b$.
\end{prop}

\begin{proof}
\hchristiane{Just for inside this proof I explicitly wrote that we use the notation $k = |\BFk|$. But can omit for consistency.}
Let $\BFk \in \N^s$. We first assume $|\BFk| \le \log (n) / \log(b)$ and write $n=jb^{|\BFk|}+r$ with $j = \lfloor nb^{-|\BFk|} \rfloor$ and $r=n-jb^{|\BFk|}$, so  $0 \le r < b^{|\BFk|}$.
From the properties of a $(t,s)-$ sequence \cite{DiPi10,rNIE92b}, each point set of the form $P_{\ell,v} = \{\BFV_{(\ell-1)b^v+1},\ldots,\BFV _{\ell b^v}\}$ for $\ell \ge 1$ and $v\ge 1$ is a $(t,v,s)-$net in base $b$. Hence,  we can split the first $n$ points of a $(0,s)$-sequence into $j$ $(0,|\BFk|,s)-$nets  and  an additional point set with $r$ points. Each of the $j$ $(0,|\BFk|,s)-$nets contributes exactly one point to each $\BFk-$elementary interval. The last $r$ points occupy exactly $r$ of the $b^{|\BFk|}$ $\BFk-$elementary intervals, as otherwise $P_{j+1,|\BFk|}$ would not be a $(0,|\BFk|,s)$-net. Therefore $r$ $\BFk-$elementary intervals have $j+1$ points and $b^{|\BFk|}-r$ have $j$ points. Hence
\[
M_b(\BFk;P_n) = rj(j+1)+(b^{|\BFk|}-r)j(j-1)
\]
and therefore
\[
C_b(\BFk;P_n) = \frac{b^{|\BFk|}}{n(n-1)} \left(rj(j+1)+(b^{|\BFk|}-r)j(j-1)\right) =  \frac{b^k}{n(n-1)}  \left(j(n-b^{|\BFk|}+r)\right).
\]
Since $jb^{|\BFk|} \le n$ and $n-b^{|\BFk|}+r \le n-1$, we obtain that $C_b(\BFk;P_n)  \le 1$. 

If $\BFk \in \N^s$ is such that 
$|\BFk| > \log (n) / \log(b)$, then $n$ of the $b^{|\BFk|}$ $\BFk-$elementary intervals have one point and $b^{|\BFk|}-n$ have 0 points. Hence 
$M_b(\BFk;P_n)$ and therefore $C_b(\BFk;P_n)$ are both 0 in that case.
\end{proof}

\begin{remark}
The proof of Proposition \ref{prop:0scqe} relies on the fact that when $P_n$ is given by the first $n$ points of a $(0,s)-$sequence in base $b$, then for any $\BFk \in \N^s$, the number of points in two different $\BFk-$elementary intervals in base $b$ differ by at most one.
\end{remark}

\hchristiane{This is new as of October 22, not discussed when skyping.}
We can also show that the $C_b(\BFk;P_n)$ values contain more information than the parameter $t$ of a digital $(t,m,s)-$net, as  demonstrated in the following proposition.

\begin{prop}
\label{prop:Ckt}
Let $P_n$ be a digital $(t,m,s)-$net in base $b$ such that the $j$th coordinates of the net form a $(0,m,1)-$net in base $b$ for each $j=1,\ldots,s$. Then 
\hchristiane{Rewritten as per Ref 1: better proposal?}
\[
t = m - 
\max \{ \ell: \ell \le m \wedge \forall \BFk \in \mathbb{N}^s: |\BFk| = \ell \Rightarrow   C_b(\BFk;P_n) \le 1 \}.
\]
\end{prop} 

\begin{proof}
It is sufficient to show that $P_n$ is $\BFk-$equidistributed if and only if $C_b(\BFk;P_n) \le 1$. The ``only if'' statement follows from Lemma \ref{lem:formCkiffequid}.
To prove the ``if'' statement, assume that $P_n$ is not   $\BFk-$equidistributed for some $\BFk$ such that ${|\BFk|} \le m$. Then it means the rank  $r=r(\BFk)$ is such that  $r<{|\BFk|}$, and that  $M_b(\BFk;P_n) =b^{m}(b^{m-r}-1)$. Hence
\[
C_b(\BFk;P_n) = \frac{b^{|\BFk|} (b^{m}(b^{m-r}-1))}{n(n-1)} = \frac{b^{|\BFk|}(b^{m-r}-1)}{b^m-1} = \frac{b^{m+{|\BFk|}-r}-b^{|\BFk|}}{b^m-1}.
\]
Now, by assumption ${|\BFk|}-r \ge 1$, therefore 
\[
C_b(\BFk;P_n)  \ge \frac{b^{m+1}-b^{|\BFk|}}{b^m-1}. 
\]
Hence to prove that $C_b(\BFk;P_n)>1$ it is sufficient to show that 
\[
\frac{b^{m+1}-b^{|\BFk|}}{b^m-1} > 1 \Leftrightarrow b^{m+1}-b^m > b^{|\BFk|}-1 \Leftrightarrow
b^m(b-1) > b^{|\BFk|} - 1
\]
and since $b \ge 2$ and ${|\BFk|} \le m$, we get $b^m(b-1) \ge b^m > b^{|\BFk|} -1$.
\end{proof}


\subsection{A functional analysis approach}

In order to establish the decomposition and bound given in  \eqref{eq:GxywithCb}, as set out at the beginning of this section, we apply tools from functional analysis. To do so,
%
we first associate the joint pdf $\psi$ of a base$-b$ digitally scrambled point set 
\hchristiane{added ``of a ... set'' as per Ref 1}
\hjaspar{we require also that the pointset has no pair of points with a shared value in one or more of the coordinates.}\hchristiane{Made a reference to Thm 3.6 instead of repeating this assumption here}
with the vector $\psi=(\psi_\BFi)_{\BFi \in \N^s}\in\ell^\infty(\N^s)$, where $\psi_\BFi$ is the value assumed by $\psi$ on $D_\BFi^s$, as given in Theorem \ref{FormOfJointPDF}. This \emph{value vector} induces a continuous linear functional $\widehat\psi:\ell^1(\N^s)\to\C$ via the formula
\begin{equation}
\label{eq:hatpsi}
\widehat\psi(\eta):=\sum_{\BFi \in \N^s} \eta_\BFi\psi_\BFi,
\end{equation}
where $\eta = (\eta_{\BFi})_{\BFi \in \mathbb{N}^s} \in \ell^1(\N^s)$. Next, for each $\BFx,\BFy\in[0,1]^s$ we define
\[
V^s(\mathbf{x},\mathbf{y}):=(V^s_\BFi(\mathbf x,\mathbf y))_{\BFi \in \N^s}\in\ell^1(\N^s)
\]
to be the \emph{volume vector} of the region $R(\BFx,\BFy)$, and observe that
\[
\|V^s(\BFx,\BFy)\|_1=\vol(R(\BFx,\BFy)).
\]
With this notation $H(\BFx,\BFy;{}_b\tilde{P}_n)= \widehat\psi( V^s(\BFx,\BFy))$. As usual, in the special case $s=1$ we drop the exponent and write
\[
V(x,y):=(V_i(x,y))_{i=0}^\infty\in\ell^1(\N).
\]
By letting
\[
\mathcal C^s:=\Big\{\frac{V^s(\mathbf{x,y})}{\vol(R(\BFx,\BFy))}:\mathbf{x,y}\in (0,1]^s\Big\}\subseteq\ell^1(\N^s)
\]
be the set of normalized volume vectors and denoting the norm of $\widehat\psi$ over $\mathcal C^s$ by 
\[
\|\widehat\psi\|_{\mathcal C^s}:=\sup_{\eta\in\mathcal C^s}\widehat\psi(\eta)
\]
we get
\[
H(\BFx,\BFy;{}_b\tilde{P}_n)=\widehat\psi( V^s(\BFx,\BFy)) \leq\vol(R(\BFx,\BFy))\|\widehat\psi\|_{\mathcal C^s},
\]
which holds for all $\BFx,\BFy\in[0,1]^s$.  Thus, in the language of functional analysis our goal is to bound $\|\widehat\psi\|_{\mathcal C^s}$.
\hchristiane{Keep the convexity argument?}


\hchristiane{Combined the two statements together}
\hjaspar{Earlier in the paper we use $[0,1)^s$ rather than $[0,1]^s$. Is there a reason we switch here?}
\begin{defi} 
\ {}
\label{def:bminusk}
\begin{enumerate}
\item For each $\BFk\in\N^s$ we define $S^\BFk:\ell^1(\N^s)\to\ell^1(\N^s)$ to be the shift 
operator that acts on the standard basis $\{e_\BFi\}_{\BFi\in\N^s}$ according to the rule $S^\BFk e_\BFi=e_{\BFi+\BFk}$, where $e_\BFi\in\ell^1(\N^s)$ is the vector whose $\BFi$th coordinate is 1 and is otherwise 0.
%
\item \label{NotationForSmallHilbertSpaceLemma}
		Given $\BFk=(k_1,\dots,k_s)\in\N^s$ and $\BFx\in[0,1]^s$ we define 
		\[
		b^{-\mathbf k}\BFx:=(b^{-k_1}x_1,\dots,b^{-k_s}x_s).
		\]
		\end{enumerate}
\end{defi}

\begin{lemma}\label{SmallHilbertSpaceLemma}
	 Let $\BFx,\BFy\in[0,1]^s$ and $\BFk\in\N^s$. Then
	\[
	S^\BFk V^s(\BFx,\BFy)=b^{2{|\BFk|}}V^s(b^{-\mathbf k}\BFx,b^{-\mathbf k}\BFy).
	\]
	In particular $S^\BFk\mathcal C^s\subseteq \mathcal C^s$.
\end{lemma}

\begin{proof}
We start by observing $\gamma_b(b^{-k}u,b^{-k}v)=\gamma_b(u,v)+k$, for all $k \in \N$ and $u,v \in [0,1)$. From this it follows  that given $x,y \in [0,1]$, the region $R(b^{-k}x,b^{-k}y)\cap D_{i+k}$ can be obtained by scaling the region $R(x,y) \cap D_i$ by a factor of $b^{-k}$ in each coordinate. Indeed
\[
 R(b^{-k}x,b^{-k}y) \cap D_{i+k}=\{(b^{-k}u,b^{-k}v):(u,v)\in  R(x,y) \cap D_i\}.
\]
Since no pair $(u,v) \in R(b^{-k}x,b^{-k}y)$ can have less than $k$ initial common digits, we can write
\[
b^{2k} V_i(b^{-k}x,b^{-k}y) = \begin{cases} V_{i-k}(x,y) &\mbox{ if } k \le i, \\ 0  &\mbox{ otherwise.} \end{cases}
\]
Now,
\begin{align*}
S^\BFk V^s(\BFx,\BFy) &= \sum_{\BFi \in \N^s} V_{\BFi}^s(\BFx,\BFy) e_{\BFi+\BFk} 
= \sum_{\BFi \in \N^s: \BFi \ge \BFk} V_{\BFi - \BFk}^s (\BFx,\BFy) e_{\BFi} \\
&= \sum_{\BFi \in \N^s: \BFi \ge \BFk} \left( \prod_{j=1}^s V_{i_j-k_j}(x_j,y_j)\right) e_\BFi \\
&=\sum_{\BFi \in \N^s} \left( \prod_{j=1}^s b^{2k_j}V_{i_j}(b^{-k_j}x_j,b^{-k_j}y_j)\right) e_\BFi \\
&=b^{2{|\BFk|}} \sum_{\BFi \in \N^s} V_{\BFi}^s (b^{-\BFk}\BFx,b^{-\BFk} \BFy) e_\BFi = b^{2{|\BFk|}} V^s(b^{-\BFk}\BFx,b^{-\BFk} \BFy). \qedhere
\end{align*}
\end{proof}

\begin{defi}
	Let $\xi\in\ell^1(\N)$ and $\xi^s\in\ell^1(\N^s)$ be defined as
	\begin{align*}
	\xi&:=V(1,1)=(\text{Vol}(D_i))_{i=0}^\infty=(\tfrac{b-1}{b^{i+1}})_{i=0}^\infty \text{ and}\\
	\xi^s&:=V^s(\mathbf{1},\mathbf{1})=(\text{Vol}(D^s_\BFi))_{\BFi\in\N^s}=(\tfrac{(b-1)^s}{b^{s+{|\BFi\,|}}})_{\BFi\in\N^s}.
	\end{align*}
\end{defi}
By Lemma \ref{SmallHilbertSpaceLemma}, we see that if $\BFk=(k_1,\dots,k_s)$, 
then 
\[
S^\BFk\xi^s=b^{2{|\BFk|}}V^s(b^{-\BFk}\mathbf 1,b^{-\BFk}\mathbf 1)\in\mathcal C^s
\]
is the normalized volume vector of $R(b^{-\BFk}\mathbf 1, b^{-\BFk}\mathbf 1)$, i.e., the region in $[0,1)^s\times[0,1)^s$ that is the product of two copies of the elementary $\BFk-$interval anchored at $\mathbf{0}$. It is also easy to verify that the standard basis vectors can be written as \[
e_k=\tfrac{b}{b-1}S^k\xi-\tfrac{1}{b-1}S^{k+1}\xi,
\]
for any $k \ge 0$. Thus
any vector in $\eta\in\ell^1(\N)$ may be written as
\begin{equation}
    \label{eq:etadecomp}
\eta=\sum_{k=0}^\infty \eta_ke_k
=\sum_{k=0}^\infty\eta_k\frac{bS^k\xi-S^{k+1}\xi}{b-1}
=\frac{b\eta_0}{b-1}\xi+\sum_{k=1}^\infty\frac{b\eta_k-\eta_{k+1}}{b-1}S^k\xi,
\end{equation}
where convergence is understood to be with respect to the norm of $\ell^1(\N)$. 
\hchristiane{could we replace norm topology by $\ell^1$ norm}\hjaspar{Yes}). This can be viewed as an ``elementary interval decomposition'' of $\eta$, i.e., a decomposition into the normalized volume vectors of 
$R(b^{-k}, b^{-k})$.
\hchristiane{changed so we're consistent with notation used 6 lines above, but in one dimension here}
The uniqueness of this decomposition follows easily from the fact that every \hchristiane{finite subset?}\hjaspar{yes} finite subset of the $S^k\xi$'s are linearly independent. To see this we observe that if $0\leq k_1<\dots<k_n$ then $S^{k_1}\xi$ is non-zero in the $k_1$ coordinate but all the other vectors $S^{k_i}\xi$ are zero in that coordinate.

\hjaspar{I seem to worry about little details more than is necessary. Uniqueness is not strictly necessary for us but I think it should be mentioned. I think that it is obvious, but maybe it is not. It follows from the fact that any finite set of the $S^k\xi$'s are linearly independent. (If $0\leq k_1<\dots<k_n$ then $S^{k_1}\xi$ is non-zero in the $k_1$ coordinate but all the other vectors $S^{k_i}\xi$ are zero there.)} 
\hchristiane{Seems like this should be a direct consequence of the fact that the $e_k$'s are a Schauder basis.}\hjaspar{This is part of the definition of a Schauder basis, so would need to be proved.}

In the special case where $\eta$ is equal to a volume vector $V(x,y)$, then 
by Lemma \ref{TechnicalLemma} we know the coefficients of $S^k\xi$ in the right-most sum of \eqref{eq:etadecomp} are non-negative for all $k\in\N$. That is, noting that when $\eta =V(x,y)$, then $\eta_k = V_k(x,y)$, we can decompose $V(x,y)$ as  
\begin{equation}
    \label{eq:Vxytk}
V(x,y) = \sum_{k=0}^\infty t_kS^k\xi
\end{equation}
where
\[
t_0=\frac{bV_0(x,y)}{b-1}\geq 0 \text{ and } t_k=\frac{bV_k(x,y)-V_{k-1}(x,y)}{b-1}\geq 0
\text{ for } k\geq 1,
\]
and the sum of the $t_k$'s equal the norm of $V(x,y)$. Indeed,
\[
\|V(x,y)\|_1=\Big\|\sum_{k=0}^\infty t_kS^k\xi\Big\|_1 = \sum_{k=0}^\infty t_k\Big\|S^k\xi\Big\|_1
=\sum_{k=0}^\infty t_k
\]
where the second equality follows from Lemma \ref{TechnicalLemma}, which implies $t_k \ge 0$, and the third one follows from the fact that  each $S^k\xi$ is a normalized vector.
\hchristiane{Changed this a little bit}


The next lemma shows that the volume vectors $V^s(\BFx,\BFy)$ arising from $s$-dimensional regions $R(\BFx,\BFy)$ satisfy \eqref{eq:Vxytk}, with $S^k\xi$ replaced with $S^\BFk\xi^s$. In this way, the above framework allows us to decompose the volume vectors of a region $R(\BFx,\BFy)$ as a conical combination (i.e.\ a linear combination with non-negative coefficients) of the volume vectors corresponding to the product of elementary intervals anchored at the origin.\hchristiane{Please go ahead and do this edit if you have something in mind. I wasn't sure what to write since this lemma is only dealing with $V$ and the $t_{\BFk}$'s and what I wrote at the beginning of the section was an analogy based on the $t_{\BFk}$ vs the $C_b(\BFk;P_n)$.}\hjaspar{Its good the way it is.}

\begin{lemma}\label{ConvCombRs}
\hjaspar{I changed the statement to include what the values of $t_k$ are. The statement holds for all $s$ and is easier to cite}
	Let $\mathbf x,\mathbf y\in[0,1]^s$ and for $j=1,\dots,s$ define
	\[
    t_{j,0}=\frac{bV_0(x_j,y_j)}{b-1}\text{ and } t_{j,k}=\frac{bV_k(x_j,y_j)-V_{k-1}(x_j,y_j)}{b-1}
    \text{ for } k\geq 1.
    \]
    Let $t_{\BFk}:=\prod_{j=1}^st_{j,k}$. Then $t_{\BFk} \ge 0$,
	\[
	V^s(\mathbf x,\mathbf y)=\sum_{\BFk \in \N^s} t_{\BFk}S^{\BFk}\xi^s, \text{ and } \sum_{\BFk \in \N^s} t_\BFk=\vol(R(\BFx,\BFy)).
	\]
\end{lemma}

\begin{proof}
In this proof we identify elements $(\eta_i)_{i=0}^{\infty} \in \ell^1(\N)$ and $(\eta_\BFi)_{\BFi \in \N^s} \in \ell^1(\N^s)$ with the power series
\[
\sum_{i=0}^{\infty} \eta_i z^i \mbox{ and } \sum_{\BFi \in \N^s} \eta_\BFi \BFz^\BFi
\]
where $\BFz^\BFi = \prod_{j=1}^s z_j^{i_j}$. In particular, we define

	\begin{equation*}
	\begin{aligned}[c]
	f_j(z)&=\sum_{i=0}^\infty V_i(x_j,y_j)z^i,\\
	g(z)&=\sum_{i=0}^\infty \xi_i z^i,\text{ and }
	\end{aligned}
	\quad
	\begin{aligned}[c]
	 f^s(z)&=\sum_{\BFi \in \N^s}^\infty V_\BFi^s(\mathbf x,\mathbf y)\BFz^\BFi,\\
	 g^s(z)&=\sum_{\BFi \in \N^s}^\infty \xi_{\BFi}^s \BFz^\BFi, 
	\end{aligned}
	\end{equation*}
	\hchristiane{WOuld it be better to replace $V_i(1,1)$ by $\xi_i$ and $V_{\BFi}(\mathbf{1}\mathbf{1})$ by $\xi_{\BFi}^s$? Would make the sentence after the display easier to process.}\hjaspar{Seems reasonable}
	and we observe that $S^k \xi$ and $S^\BFk \xi^s$ correspond to $z^k g(z)$ and $\BFz^\BFk g^s(\BFz)$ respectively. The discussion preceding this lemma shows that  $f_j(z)=g(z)h_j(z)$, where $h_j(z) = \sum_{k=0}t_{j,k}z^k$ and $t_{j,k} \ge 0$ from Lemma \ref{TechnicalLemma}. Since $f^s(\BFz) = \prod_{j=1}^s f_j(z_j)$ and $g^s(\BFz) = \prod_{j=1}^s g(z_j)$, we have	
	\begin{equation*}
	f^s(\BFz)
	=\prod_{j=1}^sg(z_j)h_j(z_j)
		=g^s(\BFz)h^s(\BFz),
	\end{equation*}
	where $h^s(\BFz)=\sum_{\BFk \in \N^s} t_\BFk \BFz^\BFk$ is the product of the $h_j(z_j)$'s. Finally, evaluating $f^s(\BFz)$ and $g^s(\BFz)$ at $\BFz = \bf{1}$ yields
	\[
	f^s(\mathbf{1})=
	\sum_{\BFi \in \N^s} V_{\BFi}^{s}(\BFx,\BFy) = \|V^s(\BFx,\BFy)\|_1 = \vol(R(\BFx,\BFy))
	\]
	and $g^s(\bf{1})=1$, thus
	\[
	\sum_{\BFk\in\N^s} t_\BFk=g^s(\mathbf{1})h^s(\mathbf{1})=f^s(\mathbf{1})=\vol(R(\BFx,\BFy)).\qedhere
	\]
\end{proof}

Now that we have shown how to decompose the volume vector $V^s(\BFx,\BFy)$ into an elementary interval decomposition indexed by $\BFk$, we are 
ready to bring back the $C_b(\BFk;P_n)$ values and explain the relation between these values and 
the joint pdf associated with ${}_b\tilde{P}_n$.

\begin{lemma}
\label{lem:hatpsiCk}
Let $P_n$ be a point set with $n$ points in $[0,1)^s$ such that the $j$th coordinate of the points are all distinct,  for each $j=1,\ldots,s$. 
Let ${}_b\tilde{P}_n$ be the sampling scheme obtained by applying a base $b-$digital scramble to $P_n$. If $\psi(\BFx,\BFy)$ denotes the joint pdf of two distinct points randomly chosen from ${}_b\tilde{P}_n$ and  $\widehat{\psi}(\cdot)$ is the linear functional defined in \eqref{eq:hatpsi}, then
\[
\widehat{\psi}(S^{\BFk} \xi^s) = C_b(\BFk;P_n).
\]
\end{lemma}

\begin{proof}
From Theorem \ref{FormOfJointPDF}, for each $\BFk \in\N^s$ we have
\hchristiane{clarify that the sum is finite so that we don't include the $N_b(\infty)$ term?}
\begin{align*}
 \widehat{\psi}(S^\BFk\xi^s)
 &=\widehat{\psi}\left(\sum_{\BFi \in \N^s} \frac{(b-1)^s}{b^{s+{|\BFi\,|}}} e_{\BFi+\BFk}\right)
 =\sum_{\BFi \in \N^s} \psi_{\BFi+\BFk} \frac{(b-1)^s}{b^{s+{|\BFi\,|}}}\\
 &=\sum_{\BFi \in \N^s}  \frac{N_b(\BFi+\BFk; {}_b\tilde{P}_n)}{n(n-1)}\frac{b^{s+{|\BFi\,|}+{|\BFk|}}}{(b-1)^s} \frac{(b-1)^s}{b^{s+{|\BFi\,|}}}
 =\sum_{\BFi\in\N^s:\BFi \ge \BFk}\frac{b^{|\BFk|} N_b(\BFi; {}_b\tilde{P}_n)}{n(n-1)}\\
&=\frac{b^k M_b(\BFk;{}_b\tilde{P}_n)}{n(n-1)} =\frac{b^k M_b(\BFk;P_n)}{n(n-1)} \mbox{ (from Remark \ref{rem:MbNb}(v))}. \qedhere
\end{align*}
\end{proof}

\hchristiane{added}
Note that Lemmas \ref{ConvCombRs} and  \ref{lem:hatpsiCk} both make use of the same decomposition based on elementary intervals indexed by $\BFk$. By combining them, we obtain the following theorem, which was already announced at the beginning of this section.

\hchristiane{Maybe we should add $G(\BFx,\BFy)$ on the LHS since this is how we presented this at the beginning of the section}
\begin{thm}\label{thm:C}
	Let $P_n$ be a deterministic point set of size $n$ and $b \ge 2$ be an integer. Assume $P_n$ is such that the $j$th coordinate of the points are all distinct  for $j=1,\ldots,s$. 
	Let ${}_b\tilde{P}_n$ be the sampling scheme obtained by applying a base $b-$digital scramble to $P_n$ and let $\psi(\BFu,\BFv)$ be the joint pdf of two distinct points randomly chosen from ${}_b\tilde{P}_n$. 
	Let $t_{\BFk}$ be the coefficient defined in Lemma \ref{ConvCombRs} for a given $\BFx,\BFy \in [0,1]^s$. Then
	\[
	\int_{R(\BFx,\BFy)}\psi(\BFu,\BFv)d\BFu d\BFv = \sum_{\BFk \in \N^s} t_{\BFk}\, C_b(\BFk;P_n)
	\]
	In particular,
	\begin{equation}
	\label{eq:BoundIntPsi}
	\int_{R(\BFx,\BFy)}\psi(\BFu,\BFv)d\BFu d\BFv  \le \vol(R(\BFx,\BFy)) \max_{\BFk \in \N^s} C_b(\BFk;P_n).
	\end{equation}
\end{thm}

\begin{proof}
As stated in  Lemma \ref{ConvCombRs},
the coefficients $t_{\BFk}$ satisfy 
$t_\BFk\geq 0$, $\sum_{\BFk \in \N^s} t_\BFk=\vol(R(\BFx,\BFy))$ and  $V^s(\mathbf{x,y}) =\sum_{\BFk \in \N^s} t_\BFk S^\BFk\xi^s$. Now
\begin{align*}
\int_{R(\BFx,\BFy)}\psi(\BFu,\BFv)d\BFu d\BFv
&= \widehat{\psi}(V^s(\mathbf{x,y}))=\widehat{\psi}\left( \sum_{\BFk \in \N^s} t_\BFk S^\BFk\xi^s\right) \\
&=\sum_{\BFk \in \N^s} t_\BFk  \widehat{\psi} (S^\BFk\xi^s) 
= \sum_{\BFk \in \N^s} t_\BFk C_b(\BFk;P_n),
\end{align*}
where the last equality follows from Lemma \ref{lem:hatpsiCk}. The inequality \eqref{eq:BoundIntPsi} is obtained by using the fact that $t_{\BFk} \ge 0$, $\sum_{\BFk \in \N^s} t_\BFk=\vol(R(\BFx,\BFy))$, and also recalling that only a finite number of vectors $\BFk$ are such that $C_b(\BFk;P_n)>0$ because of our assumption on $P_n$ having distinct coordinates, as discussed after Definition \ref{DefCqe}.
\end{proof}

\hchristiane{This first sentence was previously singled out as a remark, but I removed it as per Ref 1: we could put it back as a Remark, or just edit this whole paragraph altogether.}
From Theorem \ref{thm:C} and as discussed when presenting Eq.\ \eqref{eq:GxywithCb} at the beginning of this section, it is clear that the quantity $C_b = \max_{\BFk \in \N^s} C_b(\BFk;P_n)$ plays an important role in determining whether or not ${}_b\tilde{P}_n$ is NLOD. This will be clarified in Theorem \ref{cor:NLODiffCle1}. 
\hchristiane{Moved, as discussed}
In particular, we note that  two nets with the same value of $t$, $m$, and $s$ may have different values for the $C_b(\BFk;P_n)$ values. 
A natural question is then: ``What characteristics of $P_n$ can be measured by the $C_b(\BFk;P_n)$ values while not being captured by the parameter $t$?''
The next result provides some answers by showing that the values $C_b(\BFk;P_n)$ can be used to differentiate two point sets with respect to their propensity for negative dependence.  More precisely, it shows  that the 
$C_b(\BFk;P_n)$ values are able to capture the difference between the two nets in their ability to keep the integral of the joint pdf small. The parameter $t$ fails to capture this difference because it aggregates too much information regarding the equidistribution properties of $P_n$.

\begin{cor}
Let $P_n$ and $P_n'$ be deterministic point sets of size $n$ in $[0,1)^s$ such that the $j$th coordinate of the points are all distinct,  for each $j=1,\ldots,s$. 
Let ${}_b\tilde{P}_n$ and ${}_b\tilde{P}_n'$ be the sampling schemes obtained by applying a base $b-$digital scramble to $P_n$ and $P_n'$, respectively. Let $\psi(\BFu,\BFv)$ and  $\psi'(\BFu,\BFv)$ be the joint pdf of two distinct points randomly chosen from ${}_b\tilde{P}_n$ and ${}_b\tilde{P}_n'$, respectively. Then the following are equivalent:
\begin{enumerate}
\item For all $\BFx,\BFy \in [0,1]^s$, 
\[
	\int_{R(\BFx,\BFy)}\psi(\BFu,\BFv)d\BFu d\BFv
	\le 
	\int_{R(\BFx,\BFy)}\psi'(\BFu,\BFv)d\BFu d\BFv.
	\]
\item $C_b(\BFk;P_n) \le C_b(\BFk;P_n')$ for all $\BFk \in \N^s$.	
\end{enumerate}
\end{cor}

\begin{proof}
The fact that (2) $\Rightarrow$ (1) is easily established using Theorem \ref{thm:C}. 
\hchristiane{I included a proof, as requested by Ref 1}
To prove that (1) implies (2), 
assume there exists a $\BFk \in \N^s$ for which $C_b(\BFk;P_n) > C_b(\BFk;P_n')$. Let $\BFx=\BFy=(b^{-k_1},\ldots,b^{-k_s})$. Then 
\begin{align*}
\int_{R(\BFx,\BFy)} \psi(\BFu,\BFv)d\BFu d\BFv &=\frac{1}{b^{2{|\BFk|}}} \widehat{\psi}(S^\BFk\xi^s)
=\frac{1}{b^{2{|\BFk|}}} C_b(\BFk;P_n)\\& > \frac{1}{b^{2{|\BFk|}}}C_b(\BFk;P_n')
= \int_{R(\BFx,\BFy)} \psi'(\BFu,\BFv)d\BFu d\BFv,
\end{align*}
and thus (1) does not hold. In the above displayed equation,
the first and second equality come from Lemma \ref{SmallHilbertSpaceLemma} and Lemma \ref{lem:hatpsiCk}, respectively.
\end{proof}

We need one more technical lemma before we proceed to the next result. This technical lemma will help us show that for base $b-$digitally scrambled point sets, the NLOD and NUOD properties are equivalent.

\hchristiane{$\mathbf{0}$ is included in the domain on the RHS, but $\mathbf{1}$ is not on the RHS. NUOD def'n includes the 1, so should we also include it on the LHS? Would it then be cleaner to redefine NUOD to exclude lower bound? }

\begin{lemma}
\label{lem:flipIntPdf}
Let $P_n = \{\BFV_1,\ldots,\BFV_n\}$ be a point set such that the $j$th coordinate of the points are all distinct, for all $j=1,\ldots,s$. Let ${}_b\tilde{P}_n$ be the sampling scheme obtained by applying a base $b-$digital scramble to $P_n$. Let $\psi(\BFu,\BFv)$ be the joint pdf of two distinct points randomly chosen from ${}_b\tilde{P}_n$. Then for all $\BFx,\BFy \in [0,1]^s$ we have 
\[
\int_{[\mathbf 1-\BFx, \mathbf 1)}\int_{[\mathbf 1 -\BFy,\mathbf 1)} \psi(\BFu,\BFv)d\BFu d\BFv = \int_{R(\BFx,\BFy)} \psi(\BFu,\BFv)d\BFu d\BFv
\]
\end{lemma}

\begin{proof}
The set $A \subseteq \mathbb{R}$ containing all numbers with finite base $b$ expansion is $\{ab^{-k}:a \in \mathbb{Z},k \in \N\}$, a set of Lebesgue measure 0. If $u,v \in [0,1)\cap A^c$ have base $b$ expansions $\sum_{i=1}^{\infty} u_i b^{-i}$ and $\sum_{i=1}^{\infty} v_i b^{-i}$, respectively, then the base $b$ expansions of $1-u$ and $1-v$ are
\hchristiane{We're switching from $\psi(\BFu,\BFv)$ to $\psi(\BFx,\BFy)$. We have the $\BFx,\BFy$ in the integral domain so should really stick to $\BFu,\BFv$ for $\psi$.}\hjaspar{Okay}
\[
\sum_{i=1}^{\infty} \frac{(b-1)-u_i}{b^i} \mbox{ and } \sum_{i=1}^{\infty} \frac{(b-1)-v_i}{b^i} 
\]
respectively. It follows that $\gamma_b(u,v) = \gamma_b(1-u,1-v)$ almost everywhere and that $D_i$ is, up to a set of measure 0, invariant under the transformation $(u,v) \rightarrow (1-u,1-v)$. This means that $D_{\BFi}^s$ is also invariant, up to a set of measure 0, under the transformation $(\BFu,\BFv) \rightarrow (\mathbf 1-\BFu,\mathbf 1-\BFv)$. Because $\psi(\BFu,\BFv)$ is constant on each $D_{\BFi}^s$, $\psi(\BFu,\BFv) = \psi(\mathbf 1-\BFu,\mathbf 1-\BFv)$ except on a set of measure 0. The result then follows from integration by substitution.
\end{proof}

\hchristiane{added}
The next result gives a necessary and sufficient condition for a digitally scrambled point set to be NLOD/NUOD. The condition is based on the c.q.e.\ concept introduced in Definition \ref{DefCqe}, which holds when all $C_b(\BFk;P_n)$ values are no larger than 1.

\begin{thm}
\label{cor:NLODiffCle1}
Let $P_n$ be a deterministic point set of size $n$  in $[0,1)^s$ and $b \ge 2$ be an integer. Assume $P_n$ is such that the $j$th coordinate of the points are all distinct. Let ${}_b\tilde{P}_n$ be the sampling scheme obtained by applying a base $b-$digital scramble to $P_n$. Then ${}_b\tilde{P}_n$ is NLOD/NUOD \hchristiane{may need to stick with NLOD if haven't addressed the flip and measure 0 thing at this point} if and only if $P_n$ is c.q.e.
\end{thm}

\begin{proof}
First, 
Lemma \ref{lem:flipIntPdf} 
implies that the NLOD and NUOD properties are equivalent for point sets that have been randomized using a base $b-$digital scramble. \hjaspar{for digitially scrambled point sets.}. So we proceed to show that ${}_b\tilde{P}_n$ is NLOD if and only if $P_n$ is c.q.e.

For the ``if'' part, from Theorem \ref{thm:C}, we need to show that if $P_n$
is c.q.e.\ then
\begin{equation}
\label{eq:ineqCkiff}
\sum_{\BFk \in \N^s} t_\BFk C_b(\BFk;P_n) \le \vol(R(\BFx,\BFy))
\end{equation}
for all $\BFx,\BFy \in [0,1]^s$, recalling that the $t_{\BFk}$'s are non-negative and satisfy $\sum_{\BFk \in \N^s} t_\BFk = \vol(R(\BFx,\BFy))$.
Clearly, if $C_{b}(\BFk;P_n) \le 1$ for all $\BFk \in \N^s$ then \eqref{eq:ineqCkiff} holds.

For the ``only if'' part: assume there exists a $\BFk \in \N^s$ for which $C_b(\BFk;P_n) > 1$, and let $\BFx=\BFy=(b^{-k_1},\ldots,b^{-k_s})$. Then \[
\int_{R(\BFx,\BFy)} \psi(\BFu,\BFv)d\BFu d\BFv =\frac{1}{b^{2{|\BFk|}}} \widehat{\psi}(S^\BFk\xi^s)
=\frac{1}{b^{2{|\BFk|}}} C_b(\BFk;P_n) > \frac{1}{b^{2{|\BFk|}}} = {\rm Vol}(R(\BFx,\BFy)),
\]
where the first and second equality come from Lemma \ref{SmallHilbertSpaceLemma} and Lemma \ref{lem:hatpsiCk}, respectively. Hence ${}_b\tilde{P}_n$ is not NLOD.  
\end{proof}

We are now ready to present one of the main results of this paper.
\hchristiane{Might be able to remove digital: just need to show that if $t>0$ the $M_b(\BFk;P_n)$ will be too large for some $\BFk$}
\hjaspar{is this still the main result of the paper or is it now one of them?}
\hchristiane{Why are we not using the usual assumption, i.e, asking for $(0,m,1)$-net for the projections of $P_n$?}\hjaspar{I don't know}
\begin{thm}\label{MainTheorem}
Let ${}_b\tilde{P}_n =\{\BFU_1,\ldots,\BFU_n\}$ be a scrambled digital $(t,m,s)$-net in base $b$, with $P_n$ such that its one-dimensional projections are digital $(0,m,1)$-nets. Then ${}_b\tilde{P}_n$ is an NUOD/NLOD sampling scheme if and only if $t=0$.
\end{thm}

\hchristiane{October 22: Could shorten the proof by just referring to proposition \ref{prop:Ckt}}
\begin{proof}
As mentioned in the proof of Theorem \ref{cor:NLODiffCle1}, Lemma \ref{lem:flipIntPdf} can be used to show that the NLOD and NUOD properties are equivalent for a base $b-$digitally scrambled point set.


Next, using Theorem \ref{cor:NLODiffCle1}, it is sufficient to prove that $P_n$ is c.q.e.\ in base $b$ if and only if $t=0$. In turn, to prove the latter we use Proposition \ref{prop:Ckt}, which establishes that if $t=0$, then $C_b(\BFk;P_n) \le 1$ when ${|\BFk|} \le m$; if ${|\BFk|} > m$, then $M_b(\BFk;P_n) =0$ by  Lemma \ref{lem3p2}  so that $C_b(\BFk;P_n)=0$ as well. Proposition \ref{prop:Ckt} also establishes that if $t>0$, then there exists a $\BFk$ with ${|\BFk|} \le m$ such that $C_b(\BFk;P_n) > 1$, and therefore $P_n$ is not c.q.e.
\end{proof}

\hchristiane{Changed} Another main result is provided in the next theorem. It follows directly from applying Proposition \ref{prop:0scqe} and  Theorem \ref{cor:NLODiffCle1}. The assumption that the $j$th coordinate of the points of $P_n$ are distinct follows from the fact that any one-dimensional projection of the first $b^{\ell}$ points of a $(0,s)-$sequence is a $(0,\ell,1)-$net,  so we do not need to include this as part of our assumptions.

\begin{thm}
\label{thm:0sNLOD}
Let $P_n$ be the first $n$ points of a $(0,s)-$sequence in base $b$ and let ${}_b\tilde{P}_n$ be the point set obtained after applying a base $b-$digital scramble to $P_n$. Then ${}_b\tilde{P}_n$  is NLOD/NUOD.
\end{thm}

To end this section, we provide a result that can be directly derived from the previous theorem and the discussion in \cite{Lem17} about the class of functions for which an NUOD/NLOD sampling scheme provides an estimator with variance no larger than a Monte Carlo estimator. The interest of such a result is that it holds for any number of points rather than being given as an asymptotic bound. First, we need to introduce the following definition.

\begin{defi}
Consider a function $f:[0,1]^s \rightarrow \mathbb{R}$, and an interval of the form $A = [{\bf a},{\bf b}] = \prod_{j=1}^s [a_j,b_j] \subseteq [0,1]^s$, with $0 \le a_j \le b_j \le 1, j=1,\ldots,s$.  Let the dimension $d$ of $A$ be defined as $d=\sum_{j=1}^s {\bf 1}_{a_j<b_j}$. Let
 \[
\Delta^{(s)}(f;A)  = \sum_{{\cal I} \subseteq \{1,\ldots,s\}} (-1)^{|{\cal I}|} f(\BFa^{{\cal I}};\BFb^{-{\cal I}}),
\]
where $f(\BFa^{{\cal I}};\BFb^{-{\cal I}})$ is the function $f$ evaluated at $\BFx$ with $x_j=a_j$ if $j \in {\cal I}$ and $x_j=b_j$ if $j \notin {\cal I}$.
If  $\Delta^{(s)}(f;A) \ge 0$ for all $A$ 
of dimension $1 \le d \le s$, then $f$ is said to be {\em quasi-monotone} or {\em completely monotone}. 
\end{defi}

\begin{cor}
Let $f$ be a bounded function such that either $f$ or $-f$ is quasi-monotone. Let $P_n$ be the first $n$ points of a $(0,s)-$sequence in base $b$ and let ${}_b\tilde{P}_n$ be the point set obtained after applying a base $b-$digital scramble to $P_n$. Let $\mu_n$ be the estimator for $\mu(f)$ based on ${}_b\tilde{P}_n$ and $\hat{\mu}_{mc,n}$ the one based on a Monte Carlo estimator with $n$ points. Then ${\rm Var}(\hat{\mu}_n) \le {\rm Var}(\hat{\mu}_{mc,n})$
\end{cor}

\begin{proof}
It suffices to apply Theorem \ref{thm:0sNLOD} together with Proposition 3 from \cite{Lem17}, which says that when $f$ has a certain integral representation (see (15) in \cite{Lem17}) and $P_n$ is an NUOD sampling scheme, then the corresponding estimator has a variance no larger than a Monte Carlo estimator based on the same sample size. As mentioned in Remark 8 of \cite{Lem17}, if $f$ is bounded and either $f$ or $-f$ is quasi-monotone, then the conditions on $f$ required to apply Proposition 3 from \cite{Lem17} hold. 
\end{proof}

\hchristiane{add a result to state explicitly that the first $n$ pts of a $(0,s)-$sequence do better than MC on quasi-monotone functions?}
\hjaspar{Yes. And a remark that quasi-monotnone should really be defined as quasi-increasing. Lemma 4.17 Which I will eventually prove that any choice of orthant works, gives the quasi-monotone version of monotone rather than increasing.}

\section{Using dependence measures to assess the quality of point sets}
\label{sec:qual}

In this section, we 
highlight the potential of the quantities $C_b(\BFk;P_n)$ defined in Section \ref{sec:nlod} to be used as a flexible and informative new tool for assessing the quality of any point set, and to further our understanding of how scrambling can help improve the quality of a point set. We also note that these quantities contain information on more traditional concepts such as equidistribution in base $b$ and the $t$ parameter, as demonstrated in Lemma \ref{lem:formCkiffequid} and Proposition \ref{prop:Ckt}, but in addition they quantify the lack of equidistribution rather than simply determining if equidistribution holds or not. 

More precisely, we see three promising avenues for using these quantities  to further our understanding of low-discrepancy point sets and of the effect of scrambling. First, they can be used to predict whether or not scrambling will yield good randomized quasi-Monte Carlo estimators. Namely, the results in the previous sections show that if $C_b(\BFk;P_n) \le 1$ for all $\BFk$ then scrambling $P_n$ {\em in base $b$} will induce negative dependence, which should help reduce the variance compared to Monte Carlo sampling. We emphasize that scrambling can be performed in a base $\tilde{b}$, not necessarily equal to the base $b$ used to construct $P_n$. In particular, if $\max_{\BFk} C_b(\BFk;P_n) >1$ but $\max_{\BFk} C_{\tilde{b}}(\BFk;P_n) \le 1$,
one should consider scrambling in base $\tilde{b}$. Second, the $C_b(\BFk;P_n)$ values can be used to compare the quality of different point sets $P_n$, and could therefore be used to choose parameters for constructing $P_n$ by finding the ones that minimize a certain criterion defined by the  $C_b(\BFk;P_n)$. Indeed, a point set with smaller $C_b(\BFk;P_n)$ values not only has better equidistribution properties (before scrambling) but will result in a scrambled point set that is in some sense more negatively dependent, which should in turn result in better estimators. Third, the $C_b(\BFk;P_n)$ values allow us to compare the equidistribution properties of point sets constructed in different bases, and whose number of points is not necessarily a power of $b$.  This can in turn be used to provide key insight about why a point set seemingly better than another (say with a smaller $t$, but in a larger base) ends up not performing so well when used for integration problems.
These different avenues for further research regarding the use of the $C_b(\BFk;P_n)$ values are explored in the rest of this section using three different setups, which we now describe.


We first consider two different two-dimensional projections of a net in base 2 that are of bad quality both visually and in terms of their $t$ parameter. Both projections have $n=1024$ points and are based on a Sobol' sequence with direction numbers all set to 1. The one on the top row of Figure \ref{fig:sob} is obtained by taking the projection of that sequence over coordinates (27,28) and the one on the bottom row is obtained by  taking the projection over coordinates (22,23). \hchristiane{Ref 1 suggests to compare point sets in 3rd column of Figure 1 with MC: not sure I agree}\hjaspar{I like the idea of adding a picture of an MC point set with the same number of points. Problem is that there is no place to put it. I'll leave it to you to decide.}

\begin{figure}[htb]
\caption{Two different $(t,m,2)-$nets in base 2 with $m=10$; the middle column shows the point sets after scrambling in base 2; the right column shows the point sets after scrambling in base 53.}
\label{fig:sob}
\centering
 \begin{minipage}{0.3\textwidth}
    \setlength{\unitlength}{2.5cm}
    \begin{center}
  \includegraphics[scale =0.25]{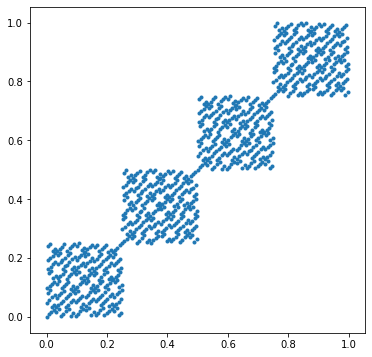}
  \end{center}
  \end{minipage}
\begin{minipage}{0.3\textwidth}
    \setlength{\unitlength}{2.5cm}
    \begin{center} 
\includegraphics[scale=0.25]{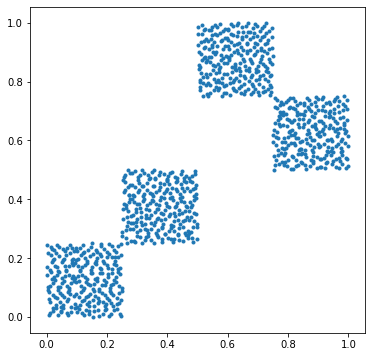}
\end{center}
\end{minipage}
\begin{minipage}{0.3\textwidth}
    \setlength{\unitlength}{2.5cm}
    \begin{center} 
\includegraphics[scale=0.25]{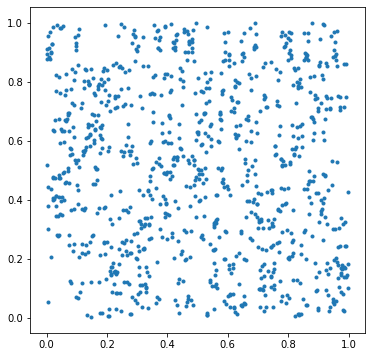}
\end{center}
\end{minipage}
 \begin{minipage}{0.3\textwidth}
    \setlength{\unitlength}{2.5cm}
    \begin{center}
  \includegraphics[scale =0.25]{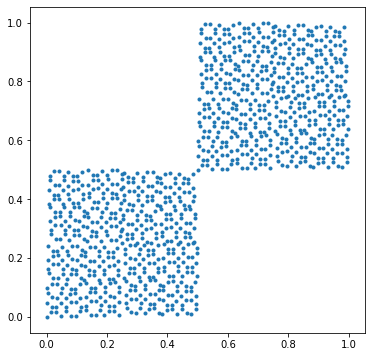}
  \end{center}
  \end{minipage}
\begin{minipage}{0.3\textwidth}
    \setlength{\unitlength}{2.5cm}
    \begin{center} 
\includegraphics[scale=0.25]{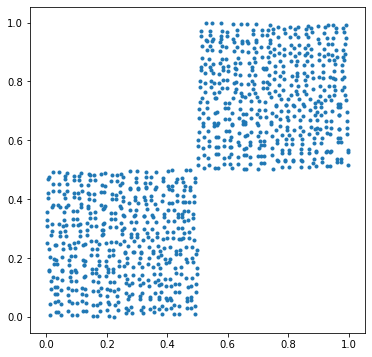}
\end{center}
\end{minipage}
\begin{minipage}{0.3\textwidth}
    \setlength{\unitlength}{2.5cm}
    \begin{center} 
\includegraphics[scale=0.25]{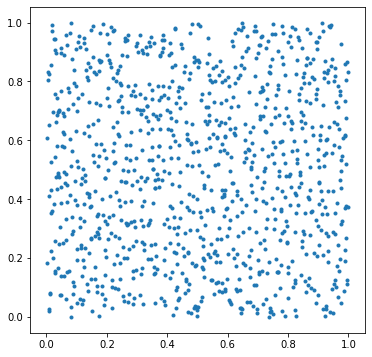}
\end{center}
\end{minipage}
\end{figure}

Table \ref{tab:Sobolb2} gives the value of $\beta_{b,k} = \max_{\BFk:|\BFk| = k} C_b(\BFk;P_n)$ for $k \ge 1$, for $b=2$. It also gives the maximum value $C_b = \max_{k \ge 1} \beta_{b,k}$, again for $b=2$.

\begin{table}[htb]
\caption{Values of $\beta_{b,k}$ and $C_b$ for $b=2$ for nets from left column of Figure \ref{fig:sob}.}
\label{tab:Sobolb2}
\begin{center}
\begin{tabular}{| l | lllllll | l |}
\hline
 $k$ & 1 & 2 & 3 & 4 & 5 & 6  & 7  \\ 
 \hline
1st pt set (top) & 1.00 & 2.00 & 1.99 & 3.99 & 3.97 & 3.94 & 3.88 \\
 2nd pt set (bottom) & 1.00 & 2.00 & 1.99 & 1.99 & 1.97 & 1.94 & 1.88 \\
 \hline
 $k$ & 8 & 9 & 10 & 11 & 12 & 13 & $C_b$ \\
 \hline
1st pt set & 3.75 & 3.50 & 3.00 & 6.01 & 4.00 &  8.01 & {\bf 8.01}  \\
 2nd pt set &1.75 &3.50 & 3.00 &6.01 &4.00 & 8.01 &{\bf 8.01}  \\ 
\hline
\end{tabular}
\end{center}
\end{table}

\begin{table}
\caption{Values of $\beta_{b,k}$ and $C_b$ for $b=53$ for nets from left column of Figure \ref{fig:sob}}
\label{tab:Sobolb53}
\begin{center}
\begin{tabular}{| l |ll | l |}
\hline
$k$ & 1 & 2 & $C_b$ \\
\hline
1st pt set & 0.95 &  2.87 & {\bf 2.87}   \\
2nd pt set & 0.95 &   0.76 & {\bf 0.95} \\
\hline
\end{tabular}
\end{center}
\end{table}

First, from Proposition \ref{prop:Ckt} we can compute
\hchristiane{seems ok because here we use $\beta_{b,k}$ but may want to update based on what we use in Prop. 4.6}
$
t = 
 m- \max\{k: \beta_{b,k} \le 1, 1 \le k \le m\}.
$
Hence from Table \ref{tab:Sobolb2} we see that in both cases, $t=9$. However, the $\beta_{b,k}$ values of the first point set are always at least as large as those for the second point set.  Corroborating this observation, we observe in Figure \ref{fig:sob} that while both point sets have large regions with no points, the design in the first point set (top left) appears to be worse than for the second one (bottom left), as we see larger contiguous empty boxes and the points are packed into a smaller region along the diagonal.

\hchristiane{Explain somewhere why $k \ge 2$ in Table 2, i.e., because $1024 \le 53^2$. I'm not updating for now in case we change some of our examples in this section.}
The plots in the centre of Figure \ref{fig:sob} show the point sets after being scrambled in base 2, using the nested uniform scrambling method of Owen \cite{vOWE95a}. Visually, we see that scrambling does not fix the issues of the deterministic point sets on the left. This is consistent with the fact that scrambling does not change the $C_b(\BFk;P_n)$ values, so if they are large in a given base, scrambling in that base will not address the lack of equidistribution. 
However when measuring $C_b(\BFk;P_n)$ in a base other than that used to construct $P_n$, 
if we find they are small (close to 1), it suggests that scrambling in that base could improve the equidistribution.  To illustrate this, we performed a base 53 scramble of the two point sets, with the resulting point sets shown on the right column of Figure \ref{fig:sob}. Visually, both point sets appear much better equidistributed after this base 53 scrambling. Note that in this case  there is no parameter $t$ that can be computed to assess the quality of ${}_{53}\tilde{P}_{n}$,  as $n=1024$ is not a power of $b$. But the $C_{53}(\BFk;P_n)$ values can be computed and are shown in Table \ref{tab:Sobolb53}. They respectively yield  a maximum $C_{53}$ of 2.87 and 0.95 for the two point sets. Hence the second scrambled point set is c.q.e.\ in base 53. Even for the first point set, $C_{53}$ is much smaller than $C_2$. Note that Table \ref{tab:Sobolb53} only reports $\beta_{b,k}$ for $k\le 2$ because $\beta_{b,k}=0$ for $k \ge 3$ for both point sets, which means that any vector $\BFk$ with $|\BFk| \ge 3$ yields $53^{|\BFk|}$ elementary intervals of size $b^{-|\BFk|}$ with either 0 or 1 point. 

This experiment shows that scrambling base 2 point sets in a larger base can be used to fix bad projections that are not repaired by the base 2 scrambling, an idea mentioned in our first avenue for exploration introduced at the beginning of this section. 
Table \ref{tab:fctc} shows the estimated variance of estimators based on these point sets for a simple integration problem, using these two different bases for the scrambling. The results confirm that the $C_b(\BFk;P_n)$ values can help predicting how successful scrambling will be at reducing the variance compared to Monte Carlo. 
\hchristiane{Following sentence added}\hjaspar{okay}

Next, we consider $(0,s)$-sequences in a prime base $b$, such as those proposed by Faure \cite{rFAU82a}. Since these sequences require $b \ge s$, in large dimensions we must work with large bases. Hence it is typical to use a number of points $n$ that is not a power of $b$. For this reason, we want to make sure the construction used is such that the first $n$ points are uniformly distributed, for any value of $n$. As discussed in, e.g.,  \cite{qLEM09a}, when working with the original Faure sequences, there can be some unwanted behavior for smaller values of $n$, i.e., smaller than $b^d$ where $d$ is the dimension of the space (or projection) considered. It is possible to construct $(0,s)$-sequences with better properties by carefully choosing deterministic scrambling matrices (often referred to as generalized Faure sequences), but it can be challenging to quantify what we mean by ``better'' since $t=0$ by definition for all these sequences, and we also know from Proposition \ref{prop:0scqe} that their first $n$ points form point sets that are all c.q.e.\ in base $b$. This is where our values $C_b(\BFk;P_n)$ can help. Figure \ref{fig:fau} shows different point sets obtained from $(0,2)$-sequences in base 53.

\hchristiane{Ref 2 is asking if distribution of scramble of Faure is same as scramble of GFaure. Could clarify this also this is not really the point of this picture.}

\begin{figure}[htb]
\caption{1024 first points of $(0,2)$-sequences taken from 49th and 50th coordinate of the following construction: original Faure sequence in base 53 (top left); generalized Faure (GFaure) sequence in base 53 obtained by randomly choosing nonsingular lower triangular matrices and multiplying them with original Faure sequence matrices (bottom left);  the middle column shows the point sets after a nested uniform scrambling in base 53; the right column shows the point sets after a nested uniform scrambling in base 2.}
\label{fig:fau}
\centering
 \begin{minipage}{0.3\textwidth}
    \setlength{\unitlength}{2.5cm}
    \begin{center}
  \includegraphics[scale =0.25]{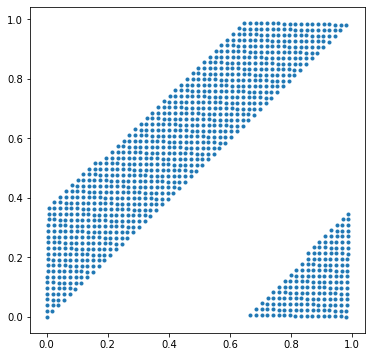}
  \end{center}
  \end{minipage}
\begin{minipage}{0.3\textwidth}
    \setlength{\unitlength}{2.5cm}
    \begin{center} 
\includegraphics[scale=0.25]{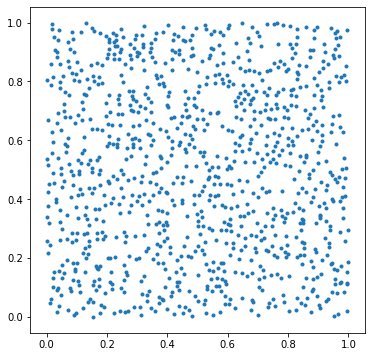}
\end{center}
\end{minipage}
\begin{minipage}{0.3\textwidth}
    \setlength{\unitlength}{2.5cm}
    \begin{center} 
\includegraphics[scale=0.25]{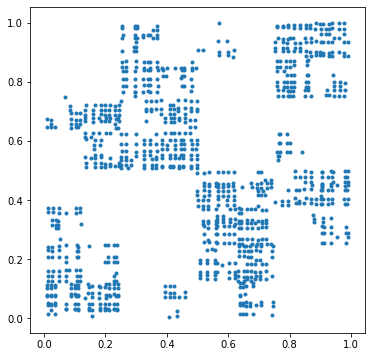}
\end{center}
\end{minipage}
\begin{minipage}{0.3\textwidth}
    \setlength{\unitlength}{2.5cm}
    \begin{center} 
\includegraphics[scale=0.25]{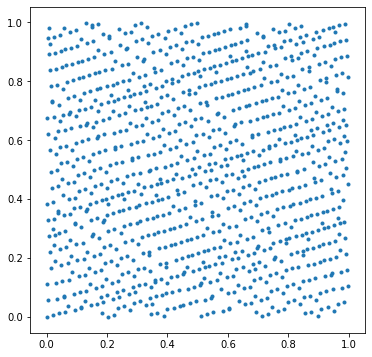}
\end{center}
\end{minipage}
 \begin{minipage}{0.3\textwidth}
    \setlength{\unitlength}{2.5cm}
    \begin{center}
  \includegraphics[scale =0.25]{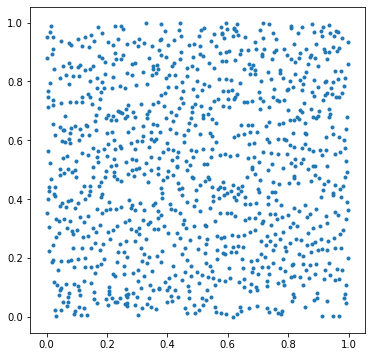}
  \end{center}
  \end{minipage}
  \begin{minipage}{0.3\textwidth}
    \setlength{\unitlength}{2.5cm}
    \begin{center} 
\includegraphics[scale=0.25]{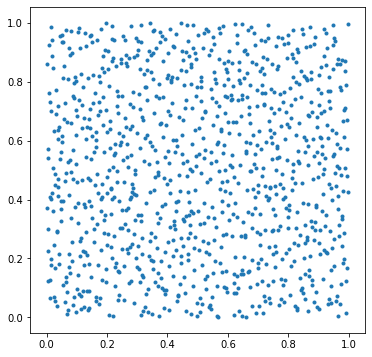}
\end{center}
\end{minipage}
\end{figure}

Since the point sets in the left column both come from the first $n=1024$ points of a $(0,2)$-sequence in base 53, they have the same values of $C_{53}(\BFk;P_n)$, namely  $\beta_{53,k}= 0.95$ for $k=1$ and is 0 otherwise. We can interpret this as follows: both point sets should have similarly  good uniformity properties after being scrambled in base 53 \hchristiane{following added} since they have the same joint pdf after being scrambled in base 53.  This is confirmed by the two figures in the middle column being very similar although the base 53 digital scrambling was applied to point sets (left side) that appear very different, the top one having much less desirable uniformity than the bottom one.
In other words, since both point sets in the middle column appear very uniform we can conclude (as expected from $\beta_{53,k}$) that both point sets are nicely distributed with respect to base 53.

In order to detect the difference between the two point sets in the left column, we compute the $C_2(\BFk;P_n)$ values for both. The motivation for doing this as follows: as seen in the right column of Figure \ref{fig:fau} and the middle column of Figure \ref{fig:sob}, a base 2 digital scrambling does not seem to address issues in a badly designed point set. This suggests that  scrambling in base 2 can only produce a uniform point set if the point set being scrambled is already uniform with respect to that base, and not only with respect to base 53. This is precisely what the $C_2(\BFk;P_n)$ can detect. 

Since the $C_2(\BFk; P_n)$ values capture the dependence structure of the base 2 scrambling of $P_n$ and we see that the two point sets look very different from each other after scrambling in base 2 (right column), those values should detect the difference between the point sets on the left. In other words, since the upper right point set  is not uniform even though a base 2 scrambling has been applied,  the upper left point set is not uniformly distributed with respect to base 2, thus the $C_2(\BFk; P_n)$ values for this point set should be larger. Similarly, since the lower right point set looks uniform, the lower left point set is not only uniformly distributed with respect to base 53 but also with respect to base 2. Table 3 shows the $C_2(\BFk, P_n)$ values of both point sets on the left. We see that the $C_2(\BFk;P_n)$ do indeed detect the difference we see visually in the point sets, with the top one giving $C_2 = 16.83$ and the bottom one giving $C_2 = 1.08$. This way of using the $C_2(\BFk, P_n)$ values, possibly in a different base than the one used for constructing $P_n$, illustrates well the potential of the approach mentioned in our second avenue for exploration, regarding the use of these values to choose parameters (in this case, deterministic scrambling matrices for the Faure sequence) for a given type of construction. These observations are further supported by the results in Table \ref{tab:fctc}, which gives the estimated variance of estimators based on these constructions for an integration problem. There we see that after scrambling in base 53, the two point sets yield estimators with approximately equal variance, which is consistent with the fact that their $C_{53}$ values are equal.

\hchristiane{Replaced the missing number by 0 for $k=16$ for Faure}
\begin{table}[htb]
\caption{Values of $\beta_{b,k}$ and $C_b$ for $b=2$ for point sets in left column of Figure \ref{fig:fau}}
\label{tab:C2fau}
{\small{
\begin{center}
\begin{tabular}{| l | llllllll  l |}
\hline
 $k$ & 1 & 2 & 3 & 4 & 5 & 6 & 7 & 8 & 9 \\ 
 \hline
Faure (top)& 1.00 &  1.10 & 1.44 &  1.85 & 2.03 &   2.25 & 2.33 &  2.52 & \!3.83 \\
GFaure (bottom) & 1.00 & 0.98 & 0.99 & 0.99 &  0.98 & 0.96 & 0.91 &  0.84 & \!0.78 \\
 \hline
 $k$ & 10 & 11 & 12 & 13 & 14 & 15 & 16 &   $C_b$ & \\
 \hline
Faure (top)  &5.89 &  7.91 & 11.18 &13.28 & 16.83 & 15.64 & 0     &  {\bf 16.83 } & \\
 GFaure (bottom) & 0.97 & 1.08 &  0.78 &  0.45 &  0.47 & 0.68& 0.63&  {\bf 1.08} &\\ 
 \hline
\end{tabular}
\end{center}}}
\end{table}

\hchristiane{added}
Our third comparison considers the projection over coordinates (16,17) of the first 1024 points of the Sobol' and Faure sequences, the latter being constructed in base 17, and the former based on direction numbers provided in \cite{FauLem17} for the so-called irreducible Sobol'-Nieddereiter sequences. Table \ref{tab:Cbsobvsfau} shows the $C_b(\BFk,P_n)$ values for $b$ equal to 2, 3, and 17.

\begin{figure}[htb]
\caption{1024 first points of a Sobol' (left) and Faure (right) sequence over coordinates (16,17)}
\label{fig:sobvsfau}
\centering
 \begin{minipage}{0.3\textwidth}
    \setlength{\unitlength}{2.5cm}
    \begin{center}
  \includegraphics[scale =0.25]{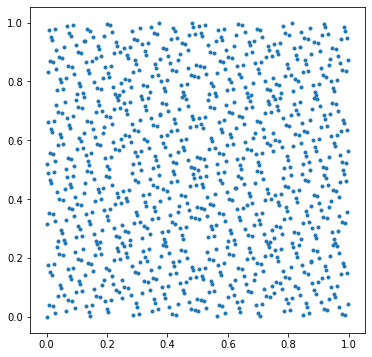}
  \end{center}
  \end{minipage}
\begin{minipage}{0.3\textwidth}
    \setlength{\unitlength}{2.5cm}
    \begin{center} 
\includegraphics[scale=0.25]{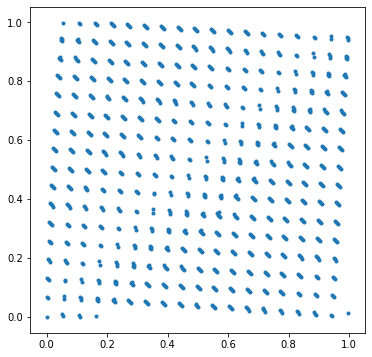}
\end{center}
\end{minipage}
\end{figure}

\begin{table}[htb]
\caption{Values of $\beta_{b,k}$ and $C_b$ for different $b$ for point sets in Figure \ref{fig:sobvsfau}}
\label{tab:Cbsobvsfau}
\begin{center}
\begin{tabular}{| l | llllllll  l |}
\hline
 $k$ & 1 & 2 & 3 & 4 & 5 & 6 & 7 & 8 & 9  \\ 
 \hline
Sobol' $(b=2)$ & 1.00&  1.00 & 0.99 &  0.99 & 0.97 & 0.94&   0.88 &  0.75 & 1.50 \\
Faure $(b=2)$  & 1.00 & 1.00 & 1.01 & 1.01 &  1.01& 1.38 & 1.48 &  2.04 & 2.43\\
Sobol' $(b=3)$ & 0.998 & 0.99 & 0.98 & 0.94 & 0.87 &
       0.74 & 0.56 & 0.18 & 0.08 \\
Faure $(b=3)$ & 1.00 &  1.01 &  1.01 &  1.27 &  1.97 & 
       2.54 &  3.23 &  6.06 &  7.37 \\
 \hline
 $k$ & 10 & 11 & 12 & 13 & 14 &15 & 16 & 17  & $C_b$  \\
 \hline
Sobol' $(b=2)$ & 1.00 & 2.00 & 0& 0&0&0&0&0& {\bf 2.00} \\
Faure $(b=2)$ & 2.76 & 3.09 & 4.93 & 6.54 & 7.29 & 3.82 & 4.00 & 0.75 & {\bf 7.29} \\
Sobol' $(b=3)$   &  0.11 & 0&0&0&0&0&0&0& {\bf 0.998}\\
Faure $(b=3)$   &  3.61 &       7.10 &0&0&0&0&0&0& {\bf 7.37}\\
  \hline
  $k$ & 1 & 2 & 3 & 4 &&&&& $C_{17}$ \\
  \hline
  Sobol' $(b=17)$ & 0.98 &  0.84 &  0.09 & 0.32 &&&&& {\bf 0.98}\\
Faure $(b=17)$ & 0.98 &  0.74 &0&0&&&&& {\bf 0.98} \\
\hline
\end{tabular}
\end{center}
\end{table}

Visually, the Faure sequence looks worse than the Sobol' sequence in Figure \ref{fig:sobvsfau}. The Sobol' point set has $t=2$ and is noticeably better than the point sets from Figure \ref{fig:sob}. We also see that it is c.q.e.\ in base 3 with $C_3=0.9980$, while for the Faure sequence $C_3=7.37$. In base 2, $C_2=2.00$ for Sobol' and $C_2=7.29$ for Faure. In base 17, both constructions have $C_{17} = 0.98$.
In other words, the base 17 equidistribution properties of the two point sets are both good but for base 2 or 3, the Sobol' point set is clearly better. One could argue that the base 2 comparison is not fair, as the Sobol' point set has been constructed in this base while the Faure one has been constructed in base 17. This is why we also included results for $C_3(\BFk;P_n)$, which confirm the superiority of the Sobol' point set for this example. This comparison highlights the potential of the  $C_b(\BFk;P_n)$ values 
to compare point sets constructed in different bases, as per the third avenue of exploration mentioned at the beginning of this section. Table \ref{tab:fctc} contains results showing that indeed the Sobol' point set performs better than the Faure one. 
More precisely, this table gives the estimated variance of  the different point sets considered in this section, with scrambling applied in different bases to obtain an estimator for the function $f(\BFu) = \prod_{j=1}^s (1+c(u_j-0.5))$ from \cite{qSOB03a}. We used 25 independent scramblings to estimate the variance in each case. Because this is a simple two-dimensional example, it would be unwise to draw too many conclusions from these results, but we can nevertheless identify a few patterns, namely that base 2 point sets with large $C_2(\BFk;P_n)$ can be ``repaired'' by a scrambling in a larger base $b$ for which their $C_b$ value is close to 1; scrambling in a base $b$ for which $C_b$ is large leads to estimators that do either worse or not much better than Monte Carlo, and point sets with small $C_b$ values in more than one base tend to yield the best estimators, regardless of the scrambling base.


\begin{table*}[htb]
\caption{Estimated variance based on scrambled point sets from Tables 1, 2 and 3}
\label{tab:fctc}
\begin{center}
\begin{tabular}{| lll |}
\hline
&\multicolumn{2}{l |}{scrambling base\,\,\,\,}\\
\hline
& $b=2$ & $b=53$   \\  
\hline
Sobol' Fig.\ 1-top & 6.10e-7 & 5.84e-9 \\
Sobol' Fig.\ 1-bottom & 6.28e-7 & 8.76e-9\\
\hline
&$b=2$  & $b=53$  \\
\hline
Faure Fig.\ 2 & 2.53e-7 &  4.31e-9\\
GFaure Fig.\ 2 &4.25e-9 & 3.86e-9\\
\hline
& $b=2$ & $b=17$ \\
\hline
Sobol' Fig.\ 3 & 6.28e-13 & 1.75e-9\\
Faure Fig.\ 3 & 8.59e-9 & 1.75e-9\\
\hline
Monte Carlo & \multicolumn{2}{c |}{3.11e-7}  \\
\hline
\end{tabular}
\end{center}
\end{table*}

To conclude this section, the main message we wish to emphasize is that the $C_b(\BFk,P_n)$ values can be very useful to assess the quality of point sets. Namely, a point set that has good overall uniformity properties should be uniformly distributed with respect to more than one base $b$ (i.e., have perfect or near equidistribution with respect to that base), which in turn should translate to small $C_b(\BFk;P_n)$ values for more than one $b$. A point set that does not possess good overall uniformity properties could exhibit small $C_b(\BFk,P_n)$ values for one base $b$, but will produce large $C_b(\BFk;P_n)$ values for some other bases $b$.
In particular, when $n$ is small relatively to $b$, it tends to be easy for $P_n$ to obtain small $C_b(\BFk;P_n)$ values in that base: in that case, a measurement in a smaller base will help detect potential issues. For such point sets, what this suggests is that their deficiencies can be repaired by a scrambling in a larger base that yields small $C_b(\BFk;P_n)$ values. If one is instead looking for a construction that does not need scrambling in order to be ``repaired'', then the $C_b(\BFk;P_n)$ values for small $b$
 can be used to choose a good-quality point set. This is illustrated in our example with $(0,s)$-sequences: as shown in Table \ref{tab:C2fau}, the GFaure construction has $C_2(\BFk;P_n)$ values that are almost all smaller than 1 and provides estimators with small variance regardless of the scrambling base. 
 On the other hand, the Faure sequence, given its larger $C_2(\BFk;P_n)$ values, should be used along with a base$-b$ digital scramble with $b=53$.

\section{Conclusion}
\label{sec:Conc}

\hchristiane{I moved some of the ideas mentioned throughout the text here. Ref is saying we are vague but I actually think we're quite precise in the description of our questions for future study}
\hchristiane{Updated on October 7 but check for relevance}
In this paper we have introduced the concept of quasi-equidistribution along with values $C_b(\BFk;P_n)$ that play a key role in analyzing the dependence structure of scrambled point sets. We have proved that scrambled digital $(0,m,s)$-nets have the property of being NUOD and NLOD and that any scrambled digital \hjaspar{Do we need to emphasise that the construction must be digital here?} \hchristiane{Added ``digital''} net with $t>0$ does not have this property. The tools we have developed to get these results will allow us to explore different paths to generalize these results. In particular, we would like to explore a generalized concept of dependence that considers sets other than the rectangular boxes anchored at the origin or at the opposite corner $(1,\ldots,1)$ that are used to define the NLOD/NUOD concepts. 
We also plan to explore how the representation for the covariance term ${\rm Cov}(f(\BFU),f(\BFV))$ as an integral of the joint pdf associated with a scrambled point set can be exploited to estimate the variance of estimators based on these point sets without having to make use of repeated randomizations.
Finally, we want to explore how the  $C_b(\BFk;P_n)$ values  can be used to construct new quality measures for digital nets.
For instance, we could combine them into a weighted measure, or summarize them differently than in Section \ref{sec:qual}, e.g., by grouping them according to which coordinates of $\BFk$ are non-zero.
 In turn, such measures could be used to design new constructions. We also want to study how the  $C_b(\BFk;P_n)$ values can be used to  assess the  propensity of scrambled nets to provide estimators with lower variance than the Monte Carlo method based on their negative dependence structure.

\section*{Acknowledgements}
The authors wish to acknowledge the support of the Natural Science and Engineering Research Council (NSERC) of Canada for its financial support via grant \# 238959. The first author is also partially supported by the Austrian Science Fund (FWF):  Projects F5506-N26 and F5509-N26, which are parts of the Special Research Program ``Quasi-Monte Carlo Methods: Theory and Applications".


\begin{thebibliography}{1}

\bibitem{DiPi10}
J.~Dick, F.~Pillichshammer, Digital Nets and Sequences: Discrepancy Theory and
  Quasi-Monte Carlo Integration, Cambridge University Press, UK, 2010.
  
\bibitem{qDIC13a}
J.~Dick, F.~Y. Kuo, I.~H. Sloan, High-dimensional integration: the
  quasi-{M}onte {C}arlo way, Acta Numerica 22 (2013) 133--288.

\bibitem{rFAU82a}
H.~Faure, Discr\'epance des suites associ\'ees \`a un syst\`eme de num\'eration
  (en dimension $s$), Acta Arithmetica 41 (1982) 337--351.
  
\bibitem{FauLem17} 
H.~Faure and C.~Lemieux. Implementation of irreducible Sobol' sequences in prime power bases,  Mathematics and Computers in Simulation 161 (2019), 13--22.

\bibitem{Ger15a}
M.~Gerber, On integration methods based on scrambled nets of arbitrary size,
  Journal of Complexity 31 (2015) 798--816.
  
\bibitem{Wnuk19Prob}
M.~Gnewuch, M.~Wnuk, N.~Hebbinghaus, On Negatively Dependent Sampling Schemes, Variance Reduction, and Probabilistic Upper Discrepancy Bounds. ArXiv preprint: 1904.10796 (2019)

\bibitem{Hic96a}  
F.~J.~Hickernell, The mean square discrepancy of randomized nets, ACM Trans. Model. Comput. Simul. 6 (1996) 274–-296.
  
\bibitem{vHON01a}
H.S. Hong and F.J. Hickernell, Algorithm 823: Implementing Scrambled Digital Sequences, ACM Trans. Math. Software 29 (2003) 95--109.

\bibitem{qLEM09a}
C.~Lemieux, Monte {C}arlo and {Q}uasi-{M}onte {C}arlo Sampling, Springer Series
  in Statistics, Springer, New York, (2009).

\bibitem{Lem17}
C.~Lemieux, Negative dependence, scrambled nets, and variance bounds,
  Mathematics of Operations Research 43 (2017) 228--251.

\bibitem{rMAT98c}
J.~Matous\v{e}k, On the ${L_2}$-discrepancy for anchored boxes, Journal of
  Complexity 14 (1998) 527--556.

\bibitem{Nelsen}
R.~Nelsen, An Introduction to Copulas, 2nd Edition, Springer Series in
  Statistics, Springer (2006).
  
\bibitem{rNIE92b}
H.~Niederreiter, Random Number Generation and Quasi-{M}onte {C}arlo Methods,
  Vol.~63 of SIAM CBMS-NSF Regional Conference Series in Applied Mathematics,
  SIAM, Philadelphia, 1992.

\bibitem{vOWE95a}
A.~B. Owen, Randomly permuted {$(t,m,s)$}-nets and {$(t,s)$}-sequences, in:
  H.~Niederreiter, P.~J.-S. Shiue (Eds.), {M}onte {C}arlo and Quasi-{M}onte
  {C}arlo Methods in Scientific Computing, Vol. 106 of Lecture Notes in
  Statistics, Springer-Verlag, New York, 1995, pp. 299--317.

\bibitem{vOWE97a}
A.~B. Owen, {M}onte {C}arlo variance of scrambled equidistribution quadrature,
  {SIAM} Journal on Numerical Analysis 34~(5) (1997) 1884--1910.
  
\bibitem{vOWE97b}
A.~B. Owen, Scrambled net variance for integrals of smooth functions, Annals of
  Statistics 25~(4) (1997) 1541--1562.

\bibitem{vOWE97c}
A.~B. Owen, Scrambling {S}obol and {N}iederreiter-{X}ing points, Journal of
  Complexity 14 (1998) 466--489.
  
\bibitem{vOWE03a}
A.~B. Owen, Variance and discrepancy with alternative scramblings, ACM
  Transactions on Modeling and Computer Simulation 13 (2003) 363--378.

\bibitem{rSOB67a}
I.~M. Sobol', On the distribution of points in a cube and the approximate
  evaluation of integrals, USSR Comp. Math. Math. Phys. 7 (1967) 86--112.
  
\bibitem{qSOB03a}
 I.~M. Sobol' and D.~I. Asotsky, One more experiment on estimating high-dimensional integrals
              by quasi-{M}onte {C}arlo methods, Math. Comput. Simul. 62 (2003) 255--263. 
 
\bibitem{Wnuk19lat}
M.~Wnuk, M.~Gnewuch. Note on pairwise negative dependence of randomized rank-1 lattices. ArXiv preprint:1903.02261, (2019).

\end{thebibliography}

\appendix
\section*{Appendix}

\begin{proof}[Proof of Lemma \ref{TechnicalLemma}]
When either $x$ or $y$ is 1, from Lemma \ref{lem:FormViOne} we know that $V_i = x (b-1)/b^{i+1}$ and thus
$bV_i-V_{i-1}=0$ in this case. So for the remainder of the proof, we assume $x,y \in [0,1)$.
\hchristiane{We need to change $k_i$ in there, especially bad with the sum over $k$ below}
Let
\[
x=\sum_{k=1}^\infty \frac{x_k}{b^k},  \text{ and } y=\sum_{k=1}^\infty \frac{y_k}{b^k}
\]
be the base $b$ digital expansion of $x$ and $y$ chosen so that only finitely many digits are non-zero. 
Recall that $k_i = \lfloor b^i \min(x,y)\rfloor b^{-i}$ for $i \ge 0$. When $\gamma_b(x,y) \ge 1$, then for $i\in \{1,\ldots,\gamma_b(x,y)\}$ we have
\[
h_i=\sum_{k=1}^i \frac{x_k}{b^k}=\sum_{k=1}^i \frac{y_k}{b^k},
\]
and $k_0=0$. We also define 
$r_x^{i}=x-h_i$, and $r_y^{i}=y-h_i$ for $i \ge 0$. Without loss of generality we assume $x\leq y$. There are four cases.

\emph{Case 1:} ($\gamma_b(x,y)<i-1$)

In this case 
\[
bV_{i}-V_{i-1}=b\frac{x}{b^{i}}-\frac{x}{b^{i-1}}=0.
\]

\emph{Case 2:} ($\gamma_b(x,y)=i-1$)

In this case, $bV_{i}-V_{i-1}$ becomes
\begin{align*}
&\frac{x}{b^{i-1}}-xy+h_{i-1}\Big(x+y-h_{i-1} -\frac{1}{b^{i-1}}\Big)\\
&=\frac{h_{i-1}+r_x^{i-1}}{b^{i-1}}-(h_{i-1}+r_x^{i-1})(h_{i-1}+r_y^{i-1})+h_{i-1}\Big(h_{i-1}+r_x^{i-1}+r_y^{i-1}-\frac{1}{b^{i-1}}\Big)\\
&=\frac{h_{i-1}+r_x^{i-1}}{b^{i-1}}-r_x^{i-1}r_y^{i-1}-\frac{h_{i-1}}{b^{i-1}}
\geq \frac{h_{i-1}+r_x^{i-1}}{b^{i-1}}-\frac{r_x^{i-1}}{b^{i-1}}-\frac{h_{i-1}}{b^{i-1}}=0
\end{align*}
because $r_y^{i-1}\leq 1/b^{i-1}$.

\emph{Case 3:} ($\gamma_b(x,y)=i$)

We use the calculation in Case 2 and the identities $r_x^{i-1}=x_{i}/b^{i}+r_x^{i}$ and $r_y^{i-1}=x_{i}/b^{i}+r_y^{i}$ to simplify $bV_{i}-V_{i-1}$:
\begin{align*}
	&b\Big(xy-\frac{x}{b^{i+1}}-k_{i}\Big(x+y-k_{i}-\frac{1}{b^{i}}\Big)\Big)\\
	&-\Big(k_{i}\Big(x+y-k_{i}-\frac{1}{b^{i}}\Big)-h_{i-1}\Big(x+y-h_{i-1} -\frac{1}{b^{i-1}}\Big)\Big)\\
	&=(b+1)\Big(xy-\frac{x}{b^{i}}-k_{i}\Big(x+y-k_{i} -\frac{1}{b^{i}}\Big)\Big)\\
	&-\Big(xy-\frac{x}{b^{i-1}}-h_{i-1}\Big(x+y-h_{i-1} -\frac{1}{b^{i-1}}\Big)\Big)\\
	&=(b+1)\Big(r_x^{i}r_y^{i}-\frac{r_x^{i}}{b^{i}}\Big)-\Big(r_x^{i-1}r_y^{i-1}-\frac{r_x^{i-1}}{b^{i-1}}\Big)\\
&=br_x^{i}r_y^{i}-\frac{r_x^{i}}{b^{i}}-\frac{x_{i}^2}{b^{2i}}-\frac{x_{i}(r_x^{i}+r_y^{i})}{b^{i}}+\frac{x_{i}}{b^{2i-1}}.
\end{align*}
Multiply by $b^{i}$ to get
\[
b^{i+1}r_x^{i}r_y^{i}-r_x^{i}-\frac{x_{i}^2}{b^{i}}-x_{i}r_x^{i}-x_{i}r_y^{i}+\frac{x_{i}}{b^{i-1}}
\]
which will be shown to be non-negative. Note that by assumption $x<y$ and since their base $b$ expansions differ for the first time at the $(i+1)^\text{th}$ digit we always have $x_{i+1}<y_{i+1}$.

\emph{Case 3a:} ($x_i\leq x_{i+1}<y_{i+1}$)

	The assumption implies $0\leq b^{i+1}r_x^i-x_i$ and  $(x_i+1)/b^{i+1}\leq r_y^i$. We estimate
\begin{align*}
&(b^{i+1}r_x^{i}-x_{i})r_y^{i}-r_x^{i}-\frac{x_{i}^2}{b^{i}}-x_{i}r_x^{i}+\frac{x_{i}}{b^{i-1}}\\
&\geq (b^{i+1}r_x^{i}-x_{i})\frac{x_{i}+1}{b^{i+1}}-r_x^{i}-\frac{x_{i}^2}{b^{i}}-x_{i}r_x^{i}+\frac{x_{i}}{b^{i-1}}\\
&=\frac{x_{i}}{b^{i+1}}(b^2-(b+1)x_{i}-1)\geq \frac{x_{i}}{b^{i+1}}(b^2-(b+1)(b-1)-1)=0.
\end{align*}

\emph{Case 3b: $(x_{i+1}<x_{i}<y_{i+1})$}

The assumption implies  $r_x^{i}\leq x_{i}/b^{i+1}$ and $(b^{i+1}r_y^{i}-x_{i}-1)\leq 0$. We estimate
\begin{align*}
&\frac{x_{i}}{b^{i-1}}-\frac{x_{i}^2}{b^{i}}-x_{i}r_y^{i}+(b^{i+1}r_y^{i}-x_{i}-1)r_x^{i}\\
&\geq \frac{x_{i}}{b^{i-1}}-\frac{x_{i}^2}{b^{i}}-x_{i}r_y^{i}+(b^{i+1}r_y^{i}-x_{i}-1)\frac{x_{i}}{b^{i+1}}\\
&=\frac{x_{i}}{b^{i+1}}(b^2-(b+1)x_{i}-1))\geq\frac{x_{i}}{b^{i+1}}(b^2-(b+1)(b-1)-1))=0.
\end{align*}

\emph{Case 3c: $(x_{i+1}<y_{i+1}\leq x_i)$}
	
	The assumption implies  $0\leq (b^{i+1}r_y^{i}-x_{i}-1)$. We estimate
\begin{align*}
&b^{i+1}r_x^{i}r_y^{i}-r_x^{i}-\frac{x_{i}^2}{b^{i}}-x_{i}r_x^{i}-x_{i}r_y^{i}+\frac{x_{i}}{b^{i-1}}=\\
&= \frac{x_{i}}{b^{i-1}}-\frac{x_{i}^2}{b^{i}}-x_{i}r_y^{i} +(b^{i+1}r_y^{i}-x_{i}-1)r_x^{i}\\
&\geq \frac{x_{i}}{b^{i-1}}-\frac{x_{i}^2}{b^{i}}-x_{i}r_y^{i}\geq x_{i}\Big(\frac{1}{b^{i-1}}-\frac{b-1}{b^{i}}-\frac{1}{b^{i}}\Big)=0.
\end{align*}

\emph{Case 4:} ($\gamma_b(x,y)>i$)

In this case we need to show that 
\begin{equation}\label{Case4}
bh_{i+1}(x+y-h_{i+1}-1/b^{i+1})-(b+1)k_{i}(x+y-k_{i}-1/b^{i})+h_{i-1}(x+y-h_{i-1} - 1/b^{i-1})
\end{equation}
is greater than or equal to zero. Using the identities $h_{i+1}=h_{i-1}+x_{i}/b^{i}+x_{i+1}/b^{i+1}$, $k_{i}=h_{i-1}+x_{i}/b^{i}$, $x=h_{i-1}+x_{i}/b^{i}+x_{i+1}/b^{i+1}+r^{i+1}_x$, and $y=h_{i-1}+x_{i}/b^{i}+x_{i+1}/b^{i+1}+r^{i+1}_y$ write
\begin{align*}
	&h_{i+1}\Big(x+y-h_{i+1}-\frac{1}{b^{i+1}}\Big)=\\
	&=\Big(h_{i-1}+\frac{x_{i}}{b^{i}}+\frac{x_{i+1}}{b^{i+1}}\Big)\Big(h_{i-1}+\frac{x_{i}}{b^{i}}+\frac{x_{i+1}}{b^{i+1}}+r^{i+1}_x+r^{i+1}_y-\frac{1}{b^{i+1}}\Big)\\
	&=h_{i-1}^2+\frac{x_{i}^2}{b^{2i}}+\frac{x_{i+1}^2}{b^{2i+2}}+\frac{2h_{i-1}x_{i}}{b^{i}}+\frac{2h_{i-1}x_{i+1}}{b^{i+1}}+\frac{2x_{i}x_{i+1}}{b^{2i+1}}+h_{i-1}(r^{i+1}_x+r^{i+1}_y)\\
	&+\frac{x_{i}(r^{i+1}_x+r^{i+1}_y)}{b^{i}}+\frac{x_{i+1}(r^{i+1}_x+r^{i+1}_y)}{b^{i+1}}-\frac{h_{i-1}}{b^{i+1}}-\frac{x_{i}}{b^{2i+1}}-\frac{x_{i+1}}{b^{2i+2}},
\end{align*}
and
\begin{align*}
	&k_{i}\Big(x+y-k_{i}-\frac{1}{b^{i}}\Big)=\\
	&=\Big(h_{i-1}+\frac{x_{i}}{b^{i}}\Big)\Big(h_{i-1}+\frac{x_{i}}{b^{i}}+\frac{2x_{i+1}}{b^{i+1}}+r^{i+1}_x+r^{i+1}_y-\frac{1}{b^{i}}\Big)\\
	&=h_{i-1}^2+\frac{x_{i}^2}{b^{2i}}+\frac{2h_{i-1}x_{i}}{b^{i}}+\frac{2h_{i-1}x_{i+1}}{b^{i+1}}+\frac{2x_{i}x_{i+1}}{b^{2i+1}}\\
	&+h_{i-1}(r^{i+1}_x+r^{i+1}_y)+\frac{x_{i}(r^{i+1}_x+ r^{i+1}_y)}{b^{i}}-\frac{h_{i-1}}{b^{i}}-\frac{x_{i}}{b^{2i}},
\end{align*}
and
\begin{align*}
&h_{i-1}\Big(x+y-h_{i-1}-\frac{1}{b^{i-1}}\Big)=\\
&=h_{i-1}\Big(h_{i-1}+\frac{2x_{i}}{b^{i}}+\frac{2x_{i+1}}{b^{i+1}}+r^{i+1}_x+r^{i+1}_y-\frac{1}{b^{i-1}}\Big)\\
&=h_{i-1}^2+\frac{2h_{i-1}x_{i}}{b^{i}}+\frac{2h_{i-1}x_{i+1}}{b^{i+1}}+h_{i-1}(r^{i+1}_x+r^{i+1}_y)-\frac{h_{i-1}}{b^{i-1}}.
\end{align*}
Now substituting into (\ref{Case4}) and simplifying \hchristiane{I wrote "simplifying" instead of "combining like terms"} we get
\begin{align*}
	&bV_{i+1}-V_{i}=0\Big(h_{i-1}^2+\frac{2h_{i-1}x_{i}}{b^{i}}+\frac{2h_{i-1}x_{i+1}}{b^{i+1}}+h_{i-1}(r^{i+1}_x+r^{i+1}_y)\Big)\\
	&-\Big(\frac{x_{i}^2}{b^{2i}}+\frac{2x_{i}x_{i+1}}{b^{2i+1}}+\frac{x_{i}(r^{i+1}_x+r^{i+1}_y)}{b^{i}}\Big)+b\Big(\frac{x_{i+1}^2}{b^{2i+2}}+\frac{x_{i+1}(r^{i+1}_x+r^{i+1}_y)}{b^{i+1}}\Big)\\
	&-h_{i-1}\Big(\frac{b}{b^{i+1}}-\frac{b+1}{b^{i}}+\frac{1}{b^{i-1}}\Big)-x_{i}\Big(\frac{b}{b^{2i+1}}-\frac{b+1}{b^{2i}}\Big)-\frac{bx_{i+1}}{b^{2i+2}}\\
	&=\frac{x_{i+1}^2}{b^{2i+1}}+\frac{x_{i}}{b^{2i-1}}-\frac{x_{i}^2}{b^{2i}}-\frac{2x_{i}x_{i+1}}{b^{2i+1}}-\frac{x_{i+1}}{b^{2i+1}}+\frac{(x_{i+1}-x_{i})(r^{i+1}_x+r^{i+1}_y)}{b^{i}}.
\end{align*}
By multiplying the above by $b^{2i+1}$ we see that to finish the proof we need to show that
\[
x_{i+1}^2+b^2x_{i}-bx_{i}^2-2x_{i}x_{i+1}-x_{i+1}+b^{i+1}(x_{i+1}-x_{i})(r^{i+1}_x+r^{i+1}_y)
\]
 is non-negative.

\emph{Case 4a:} ($x_{i}<x_{i+1}$)

We have
\begin{align*}
&x_{i+1}^2+b^2x_{i}-bx_{i}^2-2x_{i}x_{i+1}-x_{i+1}+b^{i+1}(x_{i+1}-x_{i})(r^{i+1}_x+r^{i+1}_y)\\
&\geq x_{i+1}^2+b^2x_{i}-bx_{i}^2-2x_{i}x_{i+1}-x_{i+1}\\
&\geq x_{i+1}^2+bx_{i}(x_{i+1}+1)-bx_{i}x_{i+1}-2x_{i}x_{i+1}-x_{i+1}\\
&=x_{i+1}^2+bx_{i}-2x_{i}x_{i+1}-x_{i+1}\\
&\geq x_{i+1}^2+(x_{i+1}+1)x_{i}-2x_{i}x_{i+1}-x_{i+1}\\
&\geq x_{i+1}^2-x_{i}x_{i+1}-x_{i+1}+x_{i}\\
&\geq x_{i+1}(x_{i+1}-x_{i}-1)\\
&\geq 0.
\end{align*}

\emph{Case 4b:} ($x_{i+1}\leq x_{i}$)

Since $r^{i+1}_x,r^{i+1}_y\leq 1/b^{i+1}$, we have
\begin{align*}
&x_{i+1}^2+b^2x_{i}-bx_{i}^2-2x_{i}x_{i+1}-x_{i+1}+b^{i+1}(x_{i+1}-x_{i})(r^{i+1}_x+r^{i+1}_y)\\
&\geq x_{i+1}^2+b^2x_{i}-bx_{i}^2-2x_{i}x_{i+1}+x_{i+1}-2x_{i}\\
&=(x_{i+1}-x_{i})^2+b^2x_{i}-(b+1)x_{i}^2-2x_{i}+x_{i+1}\\
&\geq (x_{i+1}-x_{i})^2+x_{i}(b^2-(b+1)(b-1)-2)+x_{i+1}\\
&=(x_{i+1}-x_{i})^2+x_{i+1}-x_{i}\\
&=(x_{i+1}-x_{i})(x_{i+1}-x_{i}+1)\\
&=(x_{i}-x_{i+1})(x_{i}-x_{i+1}-1)\geq 0.\qedhere
\end{align*}
\end{proof}

\end{document}